\documentclass{article}
\usepackage[toc,page]{appendix}
\usepackage[utf8]{inputenc}
\usepackage[english]{babel}
\DeclareMathAlphabet{\mathscr}{OT1}{pzc}{m}{it} 
\usepackage{amsmath}
\usepackage{bbold}
\usepackage{dsfont} 
\usepackage{amssymb}
\usepackage[T1]{fontenc}
\usepackage{mathtools}   
\usepackage{amsfonts}
\usepackage{bbm}
\usepackage{stmaryrd}
\usepackage{mathrsfs}  
\usepackage{graphicx}
\usepackage{bigints}
\usepackage{relsize}
\usepackage{enumitem}
\usepackage{relsize}
\usepackage[colorinlistoftodos]{todonotes}
\usepackage{amssymb,amsmath,amsthm}
\numberwithin{equation}{section}
\usepackage{authblk}

\newtheorem{theorem}{Theorem}[section]
\newtheorem{notation}[theorem]{Notation}
\newtheorem{lemma}[theorem]{Lemma}
\newtheorem{proposition}[theorem]{Proposition}
\newtheorem{corollary}[theorem]{Corollary}
\newtheorem{definition}[theorem]{Definition}
\newtheorem{hypothesis}[theorem]{Hypothesis}
\newtheorem{example}[theorem]{Example}
\newtheorem{remark}[theorem]{Remark}
\newenvironment{prooff}[1]{\begin{trivlist}
\item {\it \bf Proof}\quad} {\qed\end{trivlist}}
\usepackage{amsmath,amssymb}
\makeatletter
\newsavebox\myboxA
\newsavebox\myboxB
\newlength\mylenA
\newcommand*\xoverline[2][0.75]{%
    \sbox{\myboxA}{$\m@th#2$}%
    \setbox\myboxB\null
    \ht\myboxB=\ht\myboxA%
    \dp\myboxB=\dp\myboxA%
    \wd\myboxB=#1\wd\myboxA
    \sbox\myboxB{$\m@th\overline{\copy\myboxB}$}
    \setlength\mylenA{\the\wd\myboxA}
    \addtolength\mylenA{-\the\wd\myboxB}%
    \ifdim\wd\myboxB<\wd\myboxA%
       \rlap{\hskip 0.5\mylenA\usebox\myboxB}{\usebox\myboxA}%
    \else
        \hskip -0.5\mylenA\rlap{\usebox\myboxA}{\hskip 0.5\mylenA\usebox\myboxB}%
    \fi}
\makeatother

\title{Decoupled mild solutions of path-dependent PDEs and IPDEs
represented by BSDEs driven by cadlag martingales}

\author{
Adrien BARRASSO \thanks{ENSTA ParisTech, Unit\'e de Math\'ematiques
 appliqu\'ees, 828, boulevard des Mar\'echaux, F-91120 Palaiseau, France 
 and Ecole Polytechnique,  F-91128 Palaiseau, France.
E-mail: \sf adrien.barrasso@ensta-paristech.fr \\}
\qquad\quad
Francesco RUSSO\thanks{ENSTA ParisTech, Unit\'e de Math\'ematiques appliqu\'ees, 828, boulevard des Mar\'echaux, F-91120 Palaiseau, France. E-mail: \sf francesco.russo@ensta-paristech.fr}}

\date{April 19th 2018}

\begin{document}
\maketitle
{\bf Abstract.}
We focus on a class of path-dependent problems which include 
path-dependent (possibly Integro) PDEs, and their representation
via  BSDEs driven by a cadlag martingale.
For those equations we introduce the notion of {\it decoupled mild solution}
for which, under general assumptions, we study existence and uniqueness and its representation
via the afore mentioned BSDEs.
This concept 
generalizes a similar notion introduced by the authors in 
previous papers in the framework of classical PDEs and IPDEs.
 For every initial condition $(s,\eta)$, where $s$ is an initial time
and $\eta$ an initial path, the solution of such BSDE
 produces a couple of processes $(Y^{s,\eta},Z^{s,\eta})$.
In the classical (Markovian or not) literature
the function $u(s,\eta):= Y^{s,\eta}_s$ constitutes a 
viscosity type solution of an associated PDE (resp. IPDE);
our approach allows not only to identify $u$
as (in our language) the unique decoupled mild solution,
but also to solve quite generally the so called
 {\it identification problem}, i.e. 
to also characterize the $(Z^{s,\eta})_{s,\eta}$ processes in term of a deterministic function $v$
associated to the (above decoupled mild) solution $u$.

\bigskip
{\bf MSC 2010} Classification.  
60H30; 60H10; 35D99; 35S05; 60J35; 60J75.


\section{Introduction}

We focus on a family of path-dependent problems of the type
\begin{equation}\label{PDEIntro}
\left\{
\begin{array}{l}
A Y + f(\cdot,\cdot,Y,\Gamma(\Psi,Y))=0\text{ on }[0,T]\times\Omega\\
Y_T=\xi\text{ on }\Omega,
\end{array}\right.
\end{equation}
where $A$ is a linear map from some linear subspace $ {\mathcal D}(A)$
of the space of progressively measurable processes into the space of progressively measurable processes,
 $\Psi:=(\Psi^1,\cdots,\Psi^d)$ is a given vector of elements of $\mathcal{D}(A)$   
and $\Gamma$ is a {\it carr\'e du champs} type operator 
defined by $\Gamma(\Phi,\Phi'):=A(\Phi \Phi')-\Phi A(\Phi')-\Phi' A(\Phi)$.
Associated with this map, there is a  path-dependent
system of projectors $(P_s)_{s\in\mathbbm{R}_+}$, which extends the notion
of semigroups from the Markovian case, for which $A$ is a
weak generator, see Definition \ref{WeakGen}.
A typical example is to consider $\Psi:=X$ the canonical process, and a map $A$ given by
\begin{equation}\label{E151} 
\begin{array}{rl}
(A\Phi)_t(\omega)
:=& (D\Phi)_t(\omega) + \frac{1}{2}Tr(\sigma_t\sigma_t^{\intercal}(\nabla^2\Phi)_t(\omega))  + \beta_t(\omega)\cdot(\nabla \Phi)_t(\omega) \\
&+\int(\Phi_t(\omega+\gamma_t(\omega,y)\mathds{1}_{[t,+\infty[})-\Phi_t(\omega)- \gamma_t(\omega,y)\cdot (\nabla \Phi)_t(\omega))F(dy),
\end{array}
\end{equation}
where $\beta,\sigma,\gamma$ are bounded path-dependent predictable coefficients and $F$ is a bounded positive measure not charging $0$. In \eqref{E151}, $D$ is the horizontal derivative and $\nabla$ is the vertical gradient intended in the sense of \cite{dupire, contfournie13}.
In that case one has 
\begin{equation}
		\Gamma(X,\Phi)_t=(\sigma\sigma^{\intercal}\nabla \Phi)_t + \int_{\mathbbm{R}^d}\gamma_t(\cdot,y)(\Phi_t(\cdot+\gamma_t(\cdot,y)\mathds{1}_{[t,+\infty[})-\Phi_t)F(dy).
\end{equation}
If $\gamma\equiv 0$ then \eqref{PDEIntro} becomes the path-dependent PDE
\begin{equation}\label{PDEIntro2}
\left\{
\begin{array}{l}
DY + \frac{1}{2}Tr(\sigma\sigma^{\intercal}\nabla^2Y)  + \beta\cdot\nabla Y + f(\cdot,\cdot,Y,\sigma\sigma^{\intercal}\nabla Y)=0\text{ on }[0,T]\times\Omega\\
Y_T=\xi\text{ on }\Omega.
\end{array}\right.
\end{equation}
We introduce a notion of  {\it decoupled mild solution} 
which is inspired by the one
for classical (I)PDEs introduced in \cite{paper2,paper3},
which can be represented by solutions of Markovian BSDEs.
Concerning the corresponding  notion for \eqref{PDEIntro} 
the intuition behind  is the following.
We decouple the first line of equation \eqref{PDEIntro} into 
\begin{equation} 
\left\{
\begin{array}{ccl}
AY &=& - f(\cdot,\cdot,Y,Z)\\
Z^i &=&  \Gamma(\Psi^i,Y),\quad 1 \le i \le d,
\end{array}\right.
\end{equation}
which we can also write
\begin{equation}
\left\{
\begin{array}{ccl}
AY &=& - f(\cdot,\cdot,Y,Z)\\
A(Y \Psi^i) &=& Z^i + Y A\Psi^i + \Psi^iAY,\quad    1 \le i \le d,
\end{array}\right.
\end{equation}
and finally
\begin{equation}\label{intuition}
\left\{
\begin{array}{ccl}
AY &=& - f(\cdot,\cdot,Y,Z)\\
A(Y \Psi^i) &=& Z^i + Y A\Psi^i- \Psi^if(\cdot,\cdot,Y,Z),\quad   1 \le i \le d.
\end{array}\right.
\end{equation}
Taking \eqref{intuition} into account and inspired  by the classical notion of mild solution of an
evolution problem, we define a decoupled mild solution of equation \eqref{PDEIntro} as a functional $Y$ for which there exists an auxiliary $\mathbbm{R}^d$-valued functional $Z:=(Z^1,\cdots,Z^d)$ such that for all $(s,\eta)\in[0,T]\times \Omega$ we have
\small 
		\begin{equation}\label{MildEq}
		\left\{
		\begin{array}{rl}
		Y_s(\eta)&=P_s[\xi](\eta)+\int_s^TP_s\left[f\left(r,\cdot,Y_r,Z_r\right)\right](\eta)dr\\
		(Y\Psi^1)_s(\eta) &=P_s[\xi \Psi^1_T](\eta) -\int_s^TP_s\left[\left(Z^1_r+Y_rA\Psi^1_r-\Psi^1_rf\left(r,\cdot,Y_r,Z_r\right)\right)\right](\eta)dr\\
		&\cdots\\
		(Y\Psi^d)_s(\eta) &=P_s[\xi \Psi^d_T](\eta) -\int_s^TP_s\left[\left(Z^d_r+Y_rA\Psi^d_r-\Psi^d_rf\left(r,\cdot,Y_r,Z_r\right)\right)\right](\eta)dr.
		\end{array}\right.
		\end{equation}
\normalsize
The couple $(Y,Z)$ will be called  \textit{solution of the identification problem related to} $(f,\xi)$
because it can be strictly related to BSDEs driven by cadlag martingales which are one
natural generalization of classical Brownian BSDEs. 
We  consider for any $(s,\eta)$ the BSDE
\begin{equation}\label{BSDEIntro}
Y^{s,\eta}=\xi+\int_{\cdot}^{T}f\left(r,\cdot,Y^{s,\eta}_r,\frac{d\langle M^{s,\eta},M[\Psi]^{s,\eta}\rangle_r}{dr}\right)dr-(M^{s,\eta}_T-M^{s,\eta}_{\cdot}),
\end{equation}
in the (completed) stochastic basis $\left(\Omega,\mathcal{F}^{s,\eta},\mathbbm{F}^{s,\eta},\mathbbm{P}^{s,\eta}\right)$, where 
$(\mathbbm{P}^{s,\eta})_{(s,\eta)\in\mathbbm{R}_+\times\Omega}$ solves a martingale problem associated to $(\mathcal{D}(A),A)$.
In \eqref{BSDEIntro}, $M[\Psi]^{s,\eta}$ is the driving 
martingale of the BSDE, and is the martingale part of the process $\Psi$ under $\mathbbm{P}^{s,\eta}$.
These BSDEs were considered in a more general framework by the authors 
in \cite{paper3}. A significant contribution about BSDEs driven by 
cadlag martingales and beyond was provided by \cite{sant} and \cite{qian}.
Those BSDEs have however a {\it forward} component which is modeled
in law by the fixed family $(\mathbbm{P}^{s,\eta})_{(s,\eta)\in\mathbbm{R}_+\times\Omega}$.
An important application for path-dependent (I)PDEs is 
Theorem \ref{MainTheorem} that states the following.
 Suppose that 
 the path-dependent SDE with coefficients $\beta,\sigma,\gamma$ 
admits existence and uniqueness in law  for every initial condition $(s,\eta)$; 
we suppose moreover that $\beta_t,\sigma_t$ (resp. $\gamma_t(\cdot,x)$) are continuous for the Skorokhod topology in $\omega$ for almost all $t$ (resp. $dt\otimes dF$ a.e.),
 that  $f(\cdot,\cdot,0,0),\xi$ have  polynomial growth and that
$f$ is  Lipschitz in $(y,z)$ uniformly in $(t,\omega)$.
Then there is a unique decoupled mild solution $Y$ for \eqref{PDEIntro}
 with $\Psi:=X$ and $A$ given in \eqref{E151}. Moreover, both processes $Y,Z$ appearing in \eqref{MildEq} can be represented through the associated BSDEs \eqref{BSDEIntro}. In particular, \eqref{MildEq} gives an analytical meaning to the second process $Z$ obtained through those BSDEs.
In general the way of linking the first component $Y$ of the
 solution $(Y,Z)$ of a BSDE with 
the solution of a PDE is made by means of viscosity solutions.
However, even when the BSDE is Markovian, this does not allow to identify 
 $Z$. In particular, when $\gamma\equiv 0$,
our technique allows to characterize $Z$ as
a {\it generalized gradient}
 even if the solution does not have the vertical derivative,
contrarily to  the case in \cite{masiero}.

Brownian Backward stochastic differential equations (BSDEs)
 were introduced in  \cite{parpen90}, after a pioneering work of \cite{bismut}. When those involve a forward
dynamic described by the solution $X$ of a Brownian Markovian SDE, they are said to be Markovian, and
 are naturally linked to a parabolic PDE, see \cite{pardoux_peng92}.
In particular, under reasonable conditions,
which among others  ensure well-posedness,
the solutions of BSDEs produce {\it viscosity} type solutions
for the mentioned PDE. 
Recently 
 Brownian BSDEs of the type 
 \begin{equation} \label{EBSDE}
 	Y^{s,\eta}=\xi\left((B^{s,\eta}_t)_{t\in[0,T]}\right)+\int_{\cdot}^Tf\left(r,(B^{s,\eta}_t)_{t\in[0,r]},Y^{s,\eta}_r,Z^{s,\eta}_r\right)dr-\int_{\cdot}^TZ^{s,\eta}_rdB_r,
 \end{equation}
 where $B$ is a Brownian motion and for any $s\in[0,T]$, $\eta\in\mathbbm{D}([0,T],\mathbbm{R}^d)$, $B^{s,\eta}=\eta(\cdot\wedge s)+(B_{\cdot\vee s}-B_s)$
were associated to the path-dependent semi-linear PDE
\begin{equation}\label{PDEparabolique}
\left\{
\begin{array}{l}
D \Phi + \frac{1}{2}Tr(\nabla^2 \Phi) + f(\cdot,\cdot,\Phi,\nabla \Phi)=0\quad \text{ on } [0,T[\times\Omega \\
\Phi_T = \xi.
\end{array}\right.
\end{equation}
Path-dependent PDEs of previous type have been investigated 
by several methods. For instance strict (classical, regular) solutions 
have been 
studied in \cite{DGR, flandoli_zanco13, cosso_russo15a}
under the point of view of Banach space valued stochastic processes.
It was shown for instance in  \cite{cosso_russo15a, Peng2016} that under
some assumptions the mapping $(s,\eta)\longmapsto Y^{s,\eta}_s$ is the unique smooth solution of \eqref{PDEparabolique}.
Another popular approach is the one of {\it viscosity solutions}, 
which was considered by several authors. For instance
it was shown in \cite{ektz} 
 that if $f$ is bounded, continuous
in $t$ , uniformly continuous in the second variable, and uniformly 
Lipschitz continuous in $(y, z)$ and if $\xi$ is bounded uniformly continuous, $(s,\eta)\longmapsto Y^{s,\eta}_s$ is a viscosity solution of \eqref{PDEparabolique} 
in some specific sense, where the sense of solutions
involved the underlying probability. 
On another level, \cite{cosso_russo15b} considered
the so called {\it strong-viscosity} solutions 
(based on approximation techniques), which 
are an analytic concept, the first under non-smoothness 
conditions.  
Another interesting approach (probabilistic) but still based 
on approximation (discretizations) was given by
\cite{leao_ohashi_simas14}.
More recently, \cite{BionNadal} produced 
a viscosity solution to a more general path-dependent (possibly integro)-PDE
 through  Dynamic risk measures. 
In all those cases
the solution $\Phi$ of \eqref{PDEparabolique} was associated 
to the process $Y^{s,\eta}$ of the solution couple
$(Y^{s,\eta}, Z^{s,\eta})$ of \eqref{EBSDE} with initial time $s$ and
initial condition $\eta$.
As mentioned earlier a challenging link to be explored was 
the link between $Z^{s,\eta}$ and the solution of the path-dependent PDE $\Phi$. For instance in the case of Fr\'echet $C^{0,1}$  solutions $\Phi$ 
defined on $C([-T,0])$,
 then
 $Z^{s,\eta}$ is equal to the {\it vertical} derivative
$\nabla \Phi$, see for instance \cite{masiero}. 

The paper is organized as follows. After Section \ref{Not} 
devoted to fix some notations and basic vocabulary,
 Section \ref{S1} recalls some fundamental
 tools from the companion paper \cite{paperMPv2}.
In Section \ref{S2}, we are given  a general path dependent canonical class, 
 its associated path-dependent system of projectors $(P_s)_{s\in\mathbbm{R}_+}$
and we treat BSDEs driven by a general path-dependent MAF, 
see Definition \ref{DefAF}. 
In Subsection \ref{S2b} we are given a weak generator $A$ of $(P_s)_{s\in\mathbbm{R}_+}$,  and a corresponding abstract equation. We define the notion of decoupled mild solution of that equation and prove under some conditions, existence and uniqueness of such a solution in Theorem \ref{AbstractTheorem}.
In Section \ref{S3}, we focus on the framework of (I)PDEs. 
 In Subsection \ref{S1c} (resp. \ref{S3a}) we recall some results concerning path-dependent SDEs (resp.  path-dependent differential operators). 
In Subsection \ref{S3b}, we consider an IPDE of coefficients $\beta,\sigma,\gamma$ (which when $\gamma\equiv 0$ is given by \eqref{PDEIntro2}) and Theorem \ref{MainTheorem} states the existence and uniqueness of a decoupled mild solution.
Proposition \ref{classical} compares classical and decoupled mild solutions for that IPDE.

\section{Basic vocabulary and Notations}\label{Not}

For fixed $d,k\in\mathbbm{N}^*$, $\mathcal{C}^{k}_b(\mathbbm{R}^d)$ will denote the set of functions $k$ times differentiable with bounded continuous derivatives.
A topological space $E$ will always be considered as a measurable space 
equipped with its Borel $\sigma$-field which shall be denoted $\mathcal{B}(E)$.

Let $(\Omega,\mathcal{F})$, $(E,\mathcal{E})$ be two measurable spaces. A measurable mapping from $(\Omega,\mathcal{F})$ to $(E,\mathcal{E})$ shall often be called a \textbf{random variable} (with values in $E$), or in short r.v. 
If $\mathbbm{T}$ is some index set, a family $(X_t)_{t\in \mathbbm{T}}$  of 
 r.v. with values in $E$,  will be called  \textbf{random field} (indexed by $\mathbbm{T}$ with values in $E$). In particular, if $\mathbbm{T}$ is an interval included in $\mathbbm{R}_+$, $(X_t)_{t\in \mathbbm{T}}$ will be called a \textbf{stochastic process} (indexed by $\mathbbm{T}$ with values in $E$).
 
Given a measurable space $\left(\Omega,\mathcal{F}\right)$, for any $p \ge 1$, the set of real valued random variables with finite $p$-th moment under probability $\mathbbm{P}$ will be denoted $\mathcal{L}^p(\mathbbm{P})$ or $\mathcal{L}^p$ if there can be no ambiguity concerning the underlying probability.
Two random fields (or stochastic processes) $(X_t)_{t\in \mathbbm{T}}$, $(Y_t)_{t\in \mathbbm{T}}$ indexed by the same set and with values in the same space will be said to be \textbf{modifications (or versions)
 of each other} if for every $t\in\mathbbm{T}$, $\mathbbm{P}(X_t=Y_t)=1$.

A filtered probability space $\left(\Omega,\mathcal{F},\mathbbm{F}:=(\mathcal{F}_t)_{t\in\mathbbm{R}_+},\mathbbm{P}\right)$  will be called called  \textbf{stochastic basis} and will be said to \textbf{fulfill the usual conditions} if the filtration is right-continuous, if the probability space is complete and if $\mathcal{F}_0$ contains all the $\mathbbm{P}$-negligible sets.
Let $\left(\Omega,\mathcal{F},\mathbbm{F},\mathbbm{P}\right)$ be a stochastic
 basis.
Let $Y$ be a process and $\tau$ a stopping time, we denote  $Y^{\tau}$ the process $t\mapsto Y_{t\wedge\tau}$.
If $\mathcal{C}$ is a set of processes, 
we will say that $Y$ is {\bf locally in} $\mathcal{C}$ (resp. locally verifies some property) if there exists an a.s. increasing sequence of stopping times  $(\tau_n)_{n\geq 0}$ tending a.s. to infinity such that for every $n$, $Y^{\tau_n}$ belongs to $\mathcal{C}$ (resp. verifies the mentioned property).
In this paper we will consider martingales (with respect to a given 
filtration and probability), which are not necessarily cadlag.
For any cadlag local martingales $M,N$, we denote $[M]$ (resp. $[M,N]$) the \textbf{quadratic variation} of $M$ (resp. \textbf{quadratic covariation} of $M,N$). If moreover $M,N$ are locally square integrable, $\langle M,N\rangle$ (or simply $\langle M\rangle$ if $M=N$) will denote their (predictable) \textbf{angular bracket}.

\section{Fundamental tools}\label{S1}

\subsection{Path-dependent canonical classes and systems of projectors}\label{S1a}

We start by recalling some notions and results of Section 3 of  \cite{paperMPv2} 
that will be used all along the paper.
The first definition refers to the  canonical space that one can find 
in \cite{jacod79}, see paragraph 12.63.
\begin{notation}\label{canonicalspace}
In the whole section  $E$ will be a fixed  Polish  space (a separable completely metrizable topological space),
which will be called the \textbf{state space}.

$\Omega:=\mathbbm{D}(\mathbbm{R}_+,E)$  will denote the Skorokhod   space of functions from $\mathbbm{R}_+$ to $E$  right-continuous  with left limits (e.g. cadlag).
For every $t\in\mathbbm{R}_+$ we denote the coordinate mapping $X_t:\omega\mapsto\omega(t)$ and we define on $\Omega$ the $\sigma$-field  $\mathcal{F}:=\sigma(X_r|r\in\mathbbm{R}_+)$. 
  On the measurable space $(\Omega,\mathcal{F})$, we introduce \textbf{initial filtration}  $\mathbbm{F}^o:=(\mathcal{F}^o_t)_{t\in\mathbbm{R}_+}$, where $\mathcal{F}^o_t:=\sigma(X_r|r\in[0,t])$, and the (right-continuous) \textbf{canonical filtration} $\mathbbm{F}:=(\mathcal{F}_t)_{t\in\mathbbm{R}_+}$, where $\mathcal{F}_t:=\underset{s>t}{\bigcap}\mathcal{F}^o_s$. 
$\left(\Omega,\mathcal{F},\mathbbm{F}\right)$ will be called the \textbf{canonical space} (associated to $E$), and $X$ the \textbf{canonical process}.
On $\mathbbm{R}_+\times\Omega$, we will denote by $\mathcal{P}ro^o$ (resp. $\mathcal{P}re^o$) the  $\mathbbm{F}^o$-progressive (resp. $\mathbbm{F}^o$-predictable)  $\sigma$-field.
$\Omega$ will be equipped with the Skorokhod topology 
which makes $\Omega$ to be a Polish space  since $E$
 is itseld Polish (see Theorem 5.6 in chapter 3 of
\cite{EthierKurz}), and for which the Borel $\sigma$-field is $\mathcal{F}$, 
see Proposition 7.1 in chapter 3 of \cite{EthierKurz}.  This in particular implies that $\mathcal{F}$ is separable, as the Borel $\sigma$-field of a separable metric space.

$\mathcal{P}(\Omega)$ will denote the set of probability measures on $\Omega$ and will be equipped with the topology of weak convergence of measures which also makes it a Polish space since $\Omega$ is Polish, see Theorems 1.7 and 3.1 in \cite{EthierKurz} chapter 3. It will also be equipped with the associated Borel $\sigma$-field.
\end{notation}

\begin{notation}\label{Stopped}
For any $\omega\in\Omega$ and $t\in\mathbbm{R}_+$, the path $\omega$ stopped at time $t$ $r\mapsto \omega(r\wedge t)$ will be denoted $\omega^t$.
\end{notation}

\begin{definition}\label{DefCondSyst}
	A \textbf{path-dependent canonical class} will be a set of probability measures $(\mathbbm{P}^{s,\eta})_{(s,\eta)\in\mathbbm{R}_+\times \Omega}$ defined on the canonical space $(\Omega,\mathcal{F})$. It will verify the three following items.
	\begin{enumerate}
		\item For every  $(s,\eta)\in\mathbbm{R}_+\times \Omega$, $\mathbbm{P}^{s,\eta}( \omega^s=\eta^s)=1$;
		\item for every $s\in\mathbbm{R}_+$ and $F\in\mathcal{F}$, the mapping
		\\
		$\begin{array}{ccl}
		\eta&\longmapsto& \mathbbm{P}^{s,\eta}(F)\\
		\Omega&\longrightarrow&[0,1]
		\end{array}$ is $\mathcal{F}^o_s$-measurable;
		\item for every  $(s,\eta)\in\mathbbm{R}_+\times \Omega$, $t\geq s$  and $F\in\mathcal{F}$,
		\begin{equation} \label{DE13}
			\mathbbm{P}^{s,\eta}(F|\mathcal{F}^o_t)(\omega)=\mathbbm{P}^{t,\omega}(F)\text{ for }\mathbbm{P}^{s,\eta}\text{ almost all }\omega.
		\end{equation}
	\end{enumerate} 
This implies in particular that for every  $(s,\eta)\in\mathbbm{R}_+\times \Omega$ and $t\geq s$, then $(\mathbbm{P}^{t,\omega})_{\omega\in\Omega}$ is a regular 
conditional expectation of $\mathbbm{P}^{s,\eta}$ by $\mathcal{F}^o_t$, see the Definition above Theorem 1.1.6 in \cite{stroock} for instance.
\\
\\
A path-dependent canonical class $(\mathbbm{P}^{s,\eta})_{(s,\eta)\in\mathbbm{R}_+\times \Omega}$ will be said to be \textbf{progressive} if for every $F\in\mathcal{F}$, the mapping 
$(t,\omega)\longmapsto \mathbbm{P}^{t,\omega}(F)$ is $\mathbbm{F}^o$-progressively measurable.
\end{definition}

Very often path-dependent canonical classes will always verify the following important hypothesis which is a reinforcement of \eqref{DE13}.
\begin{hypothesis}\label{HypClass}
	For every  $(s,\eta)\in\mathbbm{R}_+\times \Omega$, $t\geq s$  and $F\in\mathcal{F}$,
	\begin{equation} \label{DE14}
	\mathbbm{P}^{s,\eta}(F|\mathcal{F}_t)(\omega)=\mathbbm{P}^{t,\omega}(F)\text{ for }\mathbbm{P}^{s,\eta}\text{ almost all }\omega.
	\end{equation}
\end{hypothesis}

\begin{remark}\label{Borel}
	By approximation through simple functions, one can easily show the
following.
	
\begin{itemize}	
\item For $s\geq 0$ and random variable $Z$ we have that $\eta \longmapsto \mathbbm{E}^{s,\eta}[Z]$ is $\mathcal{F}^o_s$-measurable and for every  $(s,\eta)\in\mathbbm{R}_+\times \Omega$, $t\geq s$,
		$\mathbbm{E}^{s,\eta}[Z|\mathcal{F}^o_t](\omega)=\mathbbm{E}^{t,\omega}[Z]\text{ for }\mathbbm{P}^{s,\eta}\text{ almost all }\omega$, provided
	previous expectations  are finite;
\item 	if the  path-dependent canonical class is progressive,
$(t,\omega)\longmapsto
	\mathbbm{E}^{t,\omega}[Z]$ is $\mathbbm{F}^o$-progressively measurable, provided
	previous expectations are finite.
\end{itemize}
\end{remark}

\begin{notation}
 $\mathcal{B}_b(\Omega)$ stands for 
the set of real bounded measurable functions on $\Omega$.
Let $s\in\mathbbm{R}_+$, $\mathcal{B}^s_b(\Omega)$  will denote the set of real bounded $\mathcal{F}^o_s$-measurable functions on $\Omega$.
	We also denote by $\mathcal{B}^+_b(\Omega)$ the subset of r.v. $\phi\in\mathcal{B}_b(\Omega)$ such that $\phi(\omega)\geq 0$ for all $\omega\in\Omega$.
\end{notation}

\begin{definition}\label{DefCondOp}
\begin{enumerate}\
\item A linear map $Q: \mathcal{B}_b(\Omega) \rightarrow \mathcal{B}_b(\Omega)$
is said {\bf positivity preserving monotonic} $\mathcal{B}_b(\Omega)$ 
 if for every $\phi\in\mathcal{B}^+_b(\Omega)$ then 
$Q[\phi]\in\mathcal{B}^+_b(\Omega)$ and 
for every increasing converging (in the pointwise sense) sequence 
$f_n\underset{n}{\longrightarrow}f$ we have  that $Q[f_n]\underset{n}{\longrightarrow}P_s[f]$ in the pointwise sense.
\item 	A family $(P_s)_{s\in\mathbbm{R}_+}$ of positivity preserving monotonic linear operators on $\mathcal{B}_b(\Omega)$ 
will be called a \textbf{path-dependent system of projectors}  if it verifies the three following items.

	\begin{itemize}
		\item For all $s\in\mathbbm{R}_+$, the restriction of $P_s$ on $\mathcal{B}^s_b(\Omega)$
		 coincides with the identity;
		\item for all $s\in\mathbbm{R}_+$, $P_s$ maps $\mathcal{B}_b(\Omega)$ 
		into $\mathcal{B}^s_b(\Omega)$;
		\item for all $s,t\in\mathbbm{R}_+$ with $t\geq s$, $P_s\circ P_t=P_s$.
	\end{itemize}
\end{enumerate}
\end{definition}

	

 The proposition below states a correspondence between path-dependent canonical classes and path-dependent systems of projectors. 
\begin{proposition}\label{EqProbaOp}
The mapping 
\begin{equation}
(\mathbbm{P}^{s,\eta})_{(s,\eta)\in\mathbbm{R}_+\times \Omega}\longmapsto\left(\begin{array}{rcl}Z&\longmapsto& (\eta\mapsto\mathbbm{E}^{s,\eta}[Z])\\ \mathcal{B}_b(\Omega)&\longrightarrow& \mathcal{B}_b(\Omega)\end{array}\right)_{s\in\mathbbm{R}_+},
\end{equation}
is a bijection between the set of path-dependent system of probability measures and the set of path-dependent system of projectors.
\end{proposition}

\begin{definition}\label{ProbaOp}
From now on, two elements in correspondence through the previous bijection will be said to be \textbf{associated}. 
\end{definition}

\begin{notation} \label{N310}
Let $(P_s)_{s\in\mathbbm{R}_+}$ be a path-dependent system of projectors, and
 $(\mathbbm{P}^{s,\eta})_{(s,\eta)\in\mathbbm{R}_+\times \Omega}$ the associated path-dependent system of probability measures.
Then for any r.v. $\phi \in \mathcal{L}^1(\mathbbm{P}^{s,\eta})$, $P_s[\phi](\eta)$ will still denote the expectation of $\phi$ under $\mathbbm{P}^{s,\eta}$. In other words we extend the linear form $\phi\longmapsto P_s[\phi](\eta)$ from $\mathcal{B}_b(\Omega)$ to $\mathcal{L}^1(\mathbbm{P}^{s,\eta})$.
\end{notation}

For the results of the whole section, we are given a progressive path-dependent canonical class $(\mathbbm{P}^{s,\eta})_{(s,\eta)\in\mathbbm{R}_+\times \Omega}$ satisfying Hypothesis \ref{HypClass} and the corresponding path-dependent system of projectors $(P_s)_{s\in\mathbbm{R}_+}$.

\begin{notation}\label{CompletedBasis}
For any $(s,\eta)\in\mathbbm{R}_+\times \Omega$ we will consider the  stochastic basis $\left(\Omega,\mathcal{F}^{s,\eta},\mathbbm{F}^{s,\eta}:=(\mathcal{F}^{s,\eta}_t)_{t\in\mathbbm{R}_+},\mathbbm{P}^{s,\eta}\right)$ where $\mathcal{F}^{s,\eta}$ (resp. $\mathcal{F}^{s,\eta}_t$ for all $t$) is $\mathcal{F}$ (resp. $\mathcal{F}_t$) augmented with  the $\mathbbm{P}^{s,\eta}$ negligible sets. $\mathbbm{P}^{s,\eta}$ is extended to $\mathcal{F}^{s,\eta}$.
\end{notation}
We remark that, for any $(s,\eta)\in\mathbbm{R}_+\times \Omega$, $\left(\Omega,\mathcal{F}^{s,\eta},\mathbbm{F}^{s,\eta},\mathbbm{P}^{s,\eta}\right)$ is a stochastic basis fulfilling the usual conditions, see 1.4 in \cite{jacod} Chapter I. \\

\begin{proposition}\label{ConditionalExp}
Let $(s,\eta)\in\mathbbm{R}_+\times \Omega$ be fixed, $Z$ be a positive r.v. or in $\mathcal{L}^1(\mathbbm{P}^{s,\eta})$ and $t\geq s$. Then 
$\mathbbm{E}^{s,\eta}[Z|\mathcal{F}_t]=\mathbbm{E}^{s,\eta}[Z|\mathcal{F}^{s,\eta}_t]$ $\mathbbm{P}^{s,\eta}$ a.s.
\end{proposition}

So  when considering conditional expectations, we will always drop the $(s,\eta)$ superscript on the filtration.
\begin{definition}\label{trivial}
Let $\mathcal{G}$ be a sub-$\sigma$-field of $\mathcal{F}$ and $\mathbbm{P}$ be a probability measure on $(\Omega,\mathcal{F})$, we say that $\mathcal{G}$ is \textbf{$\mathbbm{P}$-trivial} if for any element $G$ of $\mathcal{G}$,
 then $\mathbbm{P}(G)\in\{0,1\}$. 
\end{definition}

\begin{proposition}\label{CoroTrivial}
For every $(s,\eta)\in\mathbbm{R}_+\times\Omega$, $\mathcal{F}_s$ is $\mathbbm{P}^{s,\eta}$-trivial. In particular, an $\mathcal{F}^{s,\eta}_s$-measurable r.v. will be $\mathbbm{P}^{s,\eta}$-a.s. equal to a constant.
\end{proposition}

The last notions and results of this subsection are taken from Subsection 5.3 of \cite{paperMPv2}.

From now on we are given a non-decreasing continuous function $V$ and a couple $(\mathcal{D}(A),A)$ verifying the following.
\begin{hypothesis}\label{HypDA}
\begin{enumerate}\
\item $\mathcal{D}(A)$ is a linear subspace of the space of    $\mathbbm{F}^o$-progressively measurable processes;
\item $A$ is a linear mapping from $\mathcal{D}(A)$ into the space of  $\mathbbm{F}^o$-progressively measurable processes;
\item for all $\Phi\in\mathcal{D}(A)$, $\omega\in\Omega$, $t\geq 0$, $\int_0^t|A\Phi_r(\omega)|dV_r<+\infty$;
\item for all $\Phi\in\mathcal{D}(A)$, $(s,\eta)\in\mathbbm{R}_+\times \Omega$ and 
	$t\in[s,+\infty[$, we have  
	\\
	$\mathbbm{E}^{s,\eta}\left[\int_{s}^{t}|A(\Phi)_r|dV_r\right]<+\infty$ and $\mathbbm{E}^{s,\eta}[|\Phi_t|]<+\infty$.
\end{enumerate}
\end{hypothesis}

Inspired from the classical literature (see 13.28 in \cite{jacod}) we introduce the following notion of a weak  generator.
\begin{definition}\label{WeakGen} 
	We say that $(\mathcal{D}(A),A)$ is a \textbf{weak generator} of a path-dependent system of projectors $(P_s)_{s\in\mathbbm{R}_+}$ if for all $\Phi\in\mathcal{D}(A)$, $(s,\eta)\in\mathbbm{R}_+\times \Omega$ and 
	$t\in[s,+\infty[$, we have
	\begin{equation}
	P_s[\Phi_t](\eta)=\Phi_s(\eta)+\int_s^tP_s[A(\Phi)_r](\eta)dV_r.
	\end{equation}
\end{definition}

\begin{definition}\label{MPop}
\begin{enumerate}\
\item
	$(\mathbbm{P}^{s,\eta})_{(s,\eta)\in\mathbbm{R}_+\times \Omega}$ will be said to solve the \textbf{martingale problem associated to }$(\mathcal{D}(A),A)$ if for every $(s,\eta)\in\mathbbm{R}_+\times\Omega$, 
	\begin{itemize}
		\item $\mathbbm{P}^{s,\eta}(\omega^s=\eta^s)=1$;
		\item $\Phi-\int_0^{\cdot}A(\Phi)_rdr$,  is on $[s,+\infty[$ a $(\mathbbm{P}^{s,\eta},\mathbbm{F}^o)$-martingale
 for all $\Phi\in\mathcal{D}(A)$.
	\end{itemize} 
\item The martingale problem associated to $(\mathcal{D}(A),A)$ will be said to be \textbf{well-posed} if for every  $(s,\eta)\in\mathbbm{R}_+\times\Omega$ there exists a unique $\mathbbm{P}^{s,\eta}$ verifying both items above.
\end{enumerate}
\end{definition}

\begin{proposition}\label{MPopWellPosed}
 $(\mathcal{D}(A),A)$ is a weak generator of  $(P_s)_{s\in\mathbbm{R}_+}$ if
and only if $(\mathbbm{P}^{s,\eta})_{(s,\eta)\in\mathbbm{R}_+\times \Omega}$ solves the martingale problem associated to $(\mathcal{D}(A),A)$.
	
	In particular, if $(\mathbbm{P}^{s,\eta})_{(s,\eta)\in\mathbbm{R}_+\times \Omega}$ solves the well-posed martingale problem associated to  $(\mathcal{D}(A),A)$ then $(P_s)_{s\in\mathbbm{R}_+}$ is the unique path-dependent system of projectors for which $(\mathcal{D}(A),A)$ is a weak generator.
\end{proposition}

Indeed, the last statement allows to associate analytically  to $(\mathcal{D}(A),A)$ a unique
 path-dependent system of projectors $(P_s)_{s\in\mathbbm{R}_+}$ through Definition \ref{WeakGen}.

\subsection{Path-dependent martingale additive functionals}\label{S1b}

We now recall the notion of Path-dependent Martingale Additive Functionals
that we use in the paper. This was introduced in \cite{paperMPv2} and can
 be conceived as a path-dependent extension of the notion of non-homogeneous Martingale Additive Functionals of a Markov processes developed in  \cite{paperAF}. In this subsection, all results come from Section 4 in \cite{paperMPv2}.
 In this subsection we are again given a progressive path-dependent canonical class $(\mathbbm{P}^{s,\eta})_{(s,\eta)\in\mathbbm{R}_+\times \Omega}$ satisfying Hypothesis \ref{HypClass} and the corresponding path-dependent system of projectors $(P_s)_{s\in\mathbbm{R}_+}$.

\begin{definition}\label{DefAF}
	On $(\Omega,\mathcal{F})$, a \textbf{path-dependent Martingale Additive Functional}, in short path-dependent MAF will be  a real-valued  random-field  
	$M:=(M_{t,u})_{0\leq t\leq u}$
	verifying the two following conditions.
	\begin{enumerate}
		\item For any $0\leq t\leq u$, $M_{t,u}$ is $\mathcal{F}^o_{u}$-measurable;
		\item for any $(s,\eta)\in\mathbbm{R}_+\times \Omega$,
 there exists a real cadlag $(\mathbbm{P}^{s,\eta},\mathbbm{F}^{s,\eta})$- martingale $M^{s,\eta}$ (taken equal to zero on $[0,s]$ by convention) such that for any $\eta\in \Omega$ and $s\leq t\leq u$, 
		\begin{equation*}
		M_{t,u} = M^{s,\eta}_u-M^{s,\eta}_t \,\text{  }\, \mathbbm{P}^{s,\eta}\text{ a.s.}
		\end{equation*}
	\end{enumerate}
	
	$M^{s,\eta}$ will be called the \textbf{cadlag version of $M$ under} $\mathbbm{P}^{s,\eta}$.
	
	A path-dependent MAF will be said to verify a certain property 
	(being square integrable, having angular bracket absolutely continuous with respect to some non-decreasing function)  if under any $\mathbbm{P}^{s,\eta}$ its cadlag version verifies it.
\end{definition}

\begin{proposition}\label{CoroMaf}
Let $(\mathcal{D}(A), A)$ be a weak generator of $(P_s)_{s\in\mathbbm{R}_+}$ 
and  $(s,\eta)  \in {\mathbbm R}_+ \times \Omega.$ 
Then for every $\Phi\in\mathcal{D}(A)$, $\Phi-\int_0^{\cdot}A(\Phi)_rdV_r$ admits for all $(s,\eta)$ on $[s,+\infty[$ a $\mathbbm{P}^{s,\eta}$ version $M[\Phi]^{s,\eta}$ which is a $(\mathbbm{P}^{s,\eta},\mathbbm{F}^{s,\eta})$-cadlag martingale. In particular, the random field defined by $M[\Phi]_{t,u}(\omega):=\Phi_u(\omega)-\Phi_t(\omega)-\int_t^uA\Phi_r(\omega)dV_r$ defines a MAF with cadlag version $M[\Phi]^{s,\eta}$ under $\mathbbm{P}^{s,\eta}$.
\end{proposition}


\begin{proposition}\label{bracketMAFs}
	Let $M$ and $N$ be two square integrable path-dependent MAFs and let $M^{s,\eta}$ (respectively $N^{s,\eta}$) be the cadlag version of $M$ (respectively $N$) under a fixed $\mathbbm{P}^{s,\eta}$. Assume that $N$ has an angular bracket absolutely continuous with respect to $V$ (introduced above Hypothesis \ref{HypDA}).
%
	Then there exists an $\mathbbm{F}^o$-progressively measurable process  $k$  such that for any $(s,\eta)\in\mathbbm{R}_+\times \Omega$,
	\begin{equation*}
	\langle M^{s,\eta},N^{s,\eta}\rangle =  \int_s^{\cdot\vee s}k_rdV_r.
	\end{equation*}
\end{proposition}

\begin{notation}\label{RadonAF}
The process $k$ whose existence is stated in 
Proposition \ref{bracketMAFs}
 will be denoted $\frac{d\langle M,N\rangle}{dV}$.
\end{notation}

\section{BSDEs and abstract analytical problem}
 \label{S2}


\subsection{BSDEs driven by a path-dependent MAF}
\label{S2Prelim}

We keep using Notation \ref{canonicalspace}.
We fix a progressive path-dependent canonical class 
$(\mathbbm{P}^{s,\eta})_{(s,\eta)\in\mathbbm{R}_+\times \Omega}$  verifying Hypothesis \ref{HypClass}, and $(P_s)_{s\in\mathbbm{R}_+}$ the associated path-dependent system of projectors.
 $(\mathbbm{P}^{s,\eta})_{(s,\eta)\in\mathbbm{R}_+\times \Omega}$
will model the forward process evolution in the BSDEs.

In this section, we fix $T > 0$ and a non-decreasing continuous function 
\\$V:[0,T]\longmapsto \mathbbm{R}_+$. 
By convention, any process (resp. function) $Y$ defined on $[0,T]\times\Omega$ (resp. $[0,T]$) will be extended taking value $Y_T$ after time $T$.

\begin{notation}\label{L2uni}
For a fixed $(s,\eta)\in[0,T]\times\Omega$, we denote by $dV\otimes\mathbbm{P}^{s,\eta}$ the measure on $\mathcal{B}([s,T])\otimes\mathcal{F}$ defined by $dV\otimes\mathbbm{P}^{s,\eta}(C)=\mathbbm{E}^{s,\eta}\left[\int_s^T\mathds{1}_C(r,\omega)dV_r\right]$. For any $p\in \mathbbm{N}^*$ we denote by $\mathcal{L}^p(dV\otimes\mathbbm{P}^{s,\eta})$  the space of $(\mathcal{F}^{s,\eta}_t)_{t\in[s,T]}$-progressively measurable processes $Y$ such that $\|Y\|_{p,s,\eta}:=\left(\mathbbm{E}^{s,\eta}\left[\int_s^{T} |Y_r|^pdV_r\right]\right)^{\frac{1}{p}} < \infty$. By a slight abuse of notation we will also say that a process indexed by $[0,T]$ belongs to $\mathcal{L}^p(dV\otimes\mathbbm{P}^{s,\eta})$ if its restriction to $[s,T]\times\Omega$ does.

$\mathcal{H}^2_0(\mathbbm{P}^{s,\eta})$ will denote the space of $(\mathbbm{P}^{s,\eta},(\mathcal{F}^{s,\eta}_t)_{t\in[0,T]})$-square integrable martingales vanishing at time $s$, hence on the interval  $[0,s]$ since $\mathcal{F}_s$ is $\mathbbm{P}^{s,\eta}$-trivial, see Proposition \ref{CoroTrivial}; they will be
 defined up to indistinguishability.

For any $p\geq 1$, we define $\mathcal{L}^p_{uni}$ as the linear space of $\mathbbm{F}^o$-progressively measurable processes such that for all $(s,\eta)\in[0,T]\times\Omega$, $(Y_t)_{t\in[s,T]}$ belongs to $\mathcal{L}^p(dV\otimes d\mathbbm{P}^{s,\eta})$.
Let  $\mathcal{N}$ be the linear subspace of $\mathcal{L}^p_{uni}$ constituted of elements which are equal to $0$ $dV\otimes d\mathbbm{P}^{s,\eta}$ a.e. for all $(s,\eta)\in[0,T]\times\Omega$. We denote  $L^p_{uni}:=\mathcal{L}^p_{uni}\backslash \mathcal{N}$.
 $L^p_{uni}$ can be equipped with the topology generated by the family of semi-norms $\left(\|\cdot\|_{p,s,\eta}\right)_{(s,\eta)\in[0,T]\times \Omega}$ which makes it a separate locally convex topological vector space,
see Theorem 5.76 in \cite{aliprantis}.
\end{notation}

\begin{definition}\label{zeropotential} 
A  set $C\in \mathcal{P}ro^o$ will be said to be
 {\bf of zero potential} if $\mathds{1}_C\in\mathcal{N}$, meaning that $\mathbbm{E}^{s,\eta}\left[\int_s^{T} \mathds{1}_C(t,\omega)dV_t\right]=0$ for all $(s,\eta)$ and equivalently that $\mathds{1}_C$ is equal to $0$ in $L^2_{uni}$.
 
 A property holding everywhere in $[0,T]\times\Omega$ except on a set of zero potential will said to hold \textbf{quasi surely} abbreviated by q.s.
\end{definition}
\begin{remark}

The terminology \textit{zero potential} is inspired from classical potential theory in the Markovian setup, whereas the terminology \textit{quasi surely} comes from the theory of capacities and is justified by the fact $A$ is of zero potential iff $\underset{(s,\eta)\in[0,T]\times\Omega}{\text{sup }}dV\otimes d\mathbbm{P}^{s,\eta}(A\cap([s,T]\times\Omega))=0$.
\end{remark}

We now fix some some $d\in\mathbbm{N}^*$ and $d$  square integrable  path-dependent MAFs (see Definition \ref{DefAF}) $(N^1_{t,u})_{0\leq t\leq u},\cdots,(N^d_{t,u})_{0\leq t\leq u}$ with cadlag versions $N^{s,\eta}:=(N^{1,s,\eta},\cdots, N^{d,s,\eta})$ under $\mathbbm{P}^{s,\eta}$ for fixed $(s,\eta)$. 
\begin{definition}\label{DrivingMAF}
	$N:=(N^1,\cdots,N^d)$ will be called the \textbf{driving MAF}.
\end{definition}

In relation to this driving MAF we introduce the following hypothesis, which will be in force for the rest of the section.
\begin{hypothesis}\label{HypN}
For every 
integer $i$ such that $1 \le i \le d$, $N^i$ has an angular bracket which is absolutely continuous with respect to $V$ (see Definition \ref{DefAF}) and $\frac{d\langle N^i\rangle}{dV}$ (see Notation \ref{RadonAF}) is q.s. bounded.
\end{hypothesis}

We consider some $\xi,f$ verifying the following hypothesis.
\begin{hypothesis}\label{HypBSDE}\
	\begin{enumerate}
		\item $\xi$ is an  $\mathcal{F}_T$-measurable r.v. which belongs to $\mathcal{L}^2(\mathbbm{P}^{s,\eta})$ for every $(s,\eta)$;
		\item $f:([0,T]\times\Omega)\times\mathbbm{R}\times\mathbbm{R}^d\longmapsto \mathbbm{R}$ is measurable with respect to $\mathcal{P}ro^o\otimes\mathcal{B}(\mathbbm{R})\otimes\mathcal{B}(\mathbbm{R}^d)$ and such that 
		\begin{enumerate}
			\item $f(\cdot,\cdot,0,0)\in\mathcal{L}^2_{uni}$;
			\item there exists $K>0$ such that for all $(t,\omega,y,y',z,z')\in [0,T]\times\Omega\times\mathbbm{R}\times\mathbbm{R}\times\mathbbm{R}^d\times\mathbbm{R}^d$
			\begin{equation}
			|f(t,\omega,y',z')-f(t,\omega,y,z)|\leq K(|y'-y|+\|z'-z\|).
			\end{equation}
		\end{enumerate}
	\end{enumerate}
\end{hypothesis}

An immediate application of Theorem 3.3 and Remark 3.4 in \cite{paper3} is the following existence and uniqueness theorem.
\begin{theorem} \label{T45}
Assume the validity of  Hypotheses \ref{HypN}, \ref{HypBSDE}. For every $(s,\eta)\in[0,T]\times\Omega$, there exists a unique couple of processes $(Y^{s,\eta},M^{s,\eta})\in\mathcal{L}^2(dV\otimes\mathbbm{P}^{s,\eta})\times\mathcal{H}^2_0(\mathbbm{P}^{s,\eta})$ with $Y^{s,\eta}$ cadlag, such that on $[s,T]$
\begin{equation} \label{ET45}
Y^{s,\eta}_{\cdot}=\xi+\int_{\cdot}^{T}f\left(r,\cdot,Y^{s,\eta}_r,\frac{d\langle M^{s,\eta},N^{s,\eta}\rangle_r}{dV_r}\right)dV_r-(M^{s,\eta}_T-M^{s,\eta}_{\cdot}),
\end{equation}
in the sense of indistinguishability, under probability $\mathbbm{P}^{s,\eta}$.
\end{theorem}

\begin{notation} \label{N36}
	For the rest of this section, at fixed $(s,\eta)\in[0,T]\times\Omega$,
 the previous equation will be denoted $BSDE^{s,\eta}(f,\xi)$. Its unique solution will be denoted $(Y^{s,\eta},M^{s,\eta})$ and we will use the notation $Z^{s,\eta}:=(Z^{1,s,\eta},\cdots,Z^{d,s,\eta}) := \left(\frac{d\langle M^{s,\eta},N^{s,\eta}\rangle_t}{dV_t}\right)_{t\in[s,T]}$.
\end{notation}

\begin{remark} \label{RPicard}
	We emphasize  that equation $BSDE^{s,\eta}(f,\xi)$ of the present paper corresponds to equation $BSDE(\xi,f,V,N^{s,\eta})$ in \cite{paper3}.
\end{remark}

The following proposition can be seen as a path-dependent extension of Theorem 5.19 in \cite{paper3}. Its proof is similar to the one in the Markovian setup and is therefore postponed to the Appendix. 
\begin{proposition}\label{Defuv}  
Assume the validity of  Hypotheses \ref{HypN}, \ref{HypBSDE}.
For any $(s,\eta)$ let $(Y^{s,\eta}, M^{s,\eta})$ be as introduced at Theorem
\ref{T45}.
There exists a unique process $Y\in\mathcal{L}^2_{uni}$, a  square integrable path-dependent MAF $(M_{t,u})_{0\leq t\leq u}$ and
 $Z^1,\cdots,Z^d\in\mathcal{L}^2_{uni}$ unique up to zero potential sets (see Definition \ref{zeropotential})
such that for all $(s,\eta)\in[0,T]\times \Omega$ the following holds.
	\begin{enumerate}
		\item $Y^{s,\eta}$ is  on $[s,T]$ a $\mathbbm{P}^{s,\eta}$-modification of $Y$;
	\item $M^{s,\eta}$ is the cadlag version of $M$ under
 $\mathbbm{P}^{s,\eta}$;
\item For all integers  $ i \in \{1, \ldots,d\}$,
 $Z^i=\frac{d\langle  M^{s,\eta},N^{i,s,\eta}\rangle}{dV}$ $dV\otimes\mathbbm{P}^{s,\eta}$ a.e.
\end{enumerate}
Moreover, $Y$ is given by $Y:(s,\eta)\longmapsto Y^{s,\eta}_s$.
\end{proposition}

\subsection{Decoupled mild solutions for abstract operators}\label{S2b}
 
In this subsection, we assume that we are given some $(\mathcal{D}(A),A)$ satisfying Hypothesis \ref{HypDA} and being a weak generator of $(P_s)_{s\in\mathbbm{R}_+}$. We recall that by Proposition \ref{MPopWellPosed}, this implies that $(\mathbbm{P}^{s,\eta})_{(s,\eta)\in\mathbbm{R}_+\times \Omega}$ solves that martingale problem associated to $(\mathcal{D}(A),A)$.
By convention we will assume that every $\Phi\in\mathcal{D}(A)$ is constant after time $T$ (meaning $\Phi=\Phi_{\cdot\wedge T}$) and $A\Phi=0$ on $]T,+\infty[$.

\begin{notation} \label{N514}
For every $\Phi\in\mathcal{D}(A)$ we introduce the MAF (see Proposition \ref{CoroMaf}) 
\begin{equation}\label{MafPhi}
M[\Phi]_{t,u}(\omega):=
\Phi_u(\omega)-\Phi_t(\omega)-\int_t^uA\Phi_r(\omega)dV_r.
\end{equation}

For all $(s,\eta)$  we denote by $M[\Phi]^{s,\eta}$ its cadlag version. 

We also denote $\Phi^{s,\eta}:=\Phi_s(\eta)+\int_s^{s\vee\cdot}(A\Phi)_rdV_r + M[\Phi]^{s,\eta}$.
\end{notation}
\begin{remark}\label{RMafPhi}
By Proposition \ref{CoroMaf}, $M[\Phi]^{s,\eta}$ is a cadlag $(\mathbbm{P}^{s,\eta},\mathbbm{F}^{s,\eta})$-martingale which is a $\mathbbm{P}^{s,\eta}$-version of $\Phi_{s\vee\cdot\wedge T}-\Phi_s(\eta) -\int_s^{s\vee\cdot}(A\Phi)_rdV_r$ and therefore $\Phi^{s,\eta}$ is a cadlag special semimartingale which is a $\mathbbm{P}^{s,\eta}$-version of $\Phi$ on $[s,T]$.
\end{remark}

\begin{notation}\label{NotGamma}
	Let $\Phi,\Psi\in\mathcal{D}(A)$ be such that $\Phi\Psi\in\mathcal{D}(A)$. We denote by $\Gamma(\Phi,\Psi)$ the process $A(\Phi\Psi)-\Phi A(\Psi)-\Psi A(\Phi)$. If $\Phi$ or $\Psi$ is multidimensional, then we define 
 $\Gamma(\Phi,\Psi)$ as a vector or matrix, coordinate by coordinate.
	
	When it exists, $\Gamma(\Phi,\Phi)$ will be denoted $\Gamma(\Phi)$.
\end{notation}
$\Gamma$ can be interpreted as a path-dependent extension of the concept of carré du champ operator in the theory of Markov processes.

\begin{lemma}\label{bracketGammabis}
	Let $\Phi,\Psi\in \mathcal{D}(A)$ be such that  $\Phi\Psi\in\mathcal{D}(A)$ and assume that both $M[\Psi], M[\Phi]$ are square integrable MAFs. Then for any  $(s,\eta)\in[0,T]\times \Omega$ we have $\langle M^{s,\eta}[\Phi],M[\Psi]^{s,\eta}\rangle = \int_s^{s\vee\cdot\wedge T}\Gamma(\Phi,\Psi)dV_r$ in the sense of $\mathbbm{P}^{s,\eta}$-indistinguishability. In particular, $\Gamma(\Phi,\Psi)=\frac{d\langle M[\Phi],M[\Psi]\rangle}{dV}$ q.s., see Notation \ref{RadonAF}.
	
\end{lemma}
\begin{proof}
	We fix $\Phi,\Psi\in \mathcal{D}(A)$, $(s,\eta)\in[0,T]\times \Omega$ and the associated probability $\mathbbm{P}^{s,\eta}$.
	We recall that on $[s,T]$, $\Phi^{s,\eta} =\Phi_s(\eta)+\int_s^{\cdot}A(\Phi)_rdV_r+M^{s,\eta}[\Phi]$ and  $\Psi^{s,\eta} =
\Psi_s(\eta)+\int_s^{\cdot}A(\Psi)_r
dV_r+M^{s,\eta}[\Psi]$ are both cadlag special semi-martingales; since $M^{s,\eta}[\Phi],M^{s,\eta}[\Psi]$  are assumed to be square integrable martingales, they have a well-defined  quadratic covariation and angular bracket.
  Therefore, by integration by parts 
	on $[s,T]$ and by Lemma \ref{ModifImpliesdV} we have
	\begin{equation*} 
	\begin{array}{rl}
	\Phi^{s,\eta} \Psi^{s,\eta}=& \Phi_s(\eta)\Psi_s(\eta) + \int_s^{\cdot}(\Psi_rA(\Phi)_r+\Phi_rA(\Psi)_r)dV_r + \langle M^{s,\eta}[\Phi], M^{s,\eta}[\Psi]\rangle\\
	&+\int_s^{\cdot}\Phi^{s,\eta}_{r^-}dM^{s,\eta}[\Psi]_r +\int_s^{\cdot}\Psi^{s,\eta}_{r^-}dM^{s,\eta}[\Phi]_r\\
	&+ \left([M^{s,\eta}[\Phi],M^{s,\eta}[\Psi]]-\langle M^{s,\eta}[\Phi], M^{s,\eta}[\Psi]\rangle \right).	 
	\end{array}
	\end{equation*}
	On the other hand, since  $\Phi \Psi$ belongs to $\mathcal{D}(A)$, we also have that 
	\begin{equation}
	\Phi^{s,\eta} \Psi^{s,\eta}=(\Phi\Psi)^{s,\eta} = \Phi_s(\eta)\Psi_s(\eta) + \int_s^{\cdot}A(\Phi \Psi)_rdV_r + M^{s,\eta}[\Phi \Psi].
	\end{equation}
	By uniqueness of the decomposition of a special semi-martingale, we can identify the predictable bounded variation part in the two previous decompositions, and we get 
	\begin{equation}
	\int_s^{\cdot}(\Psi_rA(\Phi)_r+\Phi_rA(\Psi)_r)dV_r + \langle M^{s,\eta}[\Phi], M^{s,\eta}[\Psi]\rangle=\int_s^{\cdot}A(\Phi \Psi)_rdV_r,
	\end{equation}
	hence that 
	\begin{equation}
	\langle M^{s,\eta}[\Phi], M^{s,\eta}[\Psi]\rangle = \int_s^{\cdot}(A(\Phi \Psi)_r-\Psi_rA(\Phi)_r-\Phi_rA(\Psi)_r)dV_r,
	\end{equation}
	and the proof is complete.
\end{proof}
\begin{lemma}\label{bracketGammater}
If $\Phi\in\mathcal{D}(A)$, $\Phi^2\in\mathcal{D}(A)$ and $A\Phi\in\mathcal{L}^2_{uni}$, then the MAF $M[\Phi]$ is square integrable. Moreover, $\underset{t\in[s,T]}{\text{sup }}|\Phi^{s,\eta}_t|\in\mathcal{L}^2(\mathbbm{P}^{s,\eta})$ for all $(s,\eta)$.
\end{lemma}
\begin{proof}
Let $(s,\eta)$ be fixed. We have $(\Phi^{s,\eta})^2= \Phi_s^2(\eta)+ \int_s^{\cdot}A(\Phi^2)_rdV_r + M^{s,\eta}[\Phi^2]$, where $M^{s,\eta}[\Phi^2]$ is a martingale, hence takes $\mathcal{L}^1$-values; moreover
 $A(\Phi^2)\in\mathcal{L}^1(dV\otimes d\mathbbm{P}^{s,\eta})$, see Hypothesis \ref{HypDA}. It is therefore clear that $\Phi^{s,\eta}$ has 
$\mathcal{L}^2$-values. Since 
\begin{equation}\label{EbracketGammater}
\Phi^{s,\eta}=\Phi_s(\eta)+ \int_s^{\cdot}A(\Phi)_rdV_r + M^{s,\eta}[\Phi],
\end{equation}
 by Cauchy-Schwarz inequality we have
$$
  ( M^{s,\eta}[\Phi]_T)^2\leq 4(\Phi^2_s(\eta)+(V_T-V_s)\int_s^TA(\Phi)^2_rdV_r+(\Phi^{s,\eta}_T)^2).  
$$ 
The right-hand side belongs to $ \mathcal{L}^1$
 because $A\Phi\in\mathcal{L}^2_{uni}$, $\Phi^2_s(\eta)$ is deterministic and $\Phi^{s,\eta}$ takes $\mathcal{L}^2$ values, therefore
 $M^{s,\eta}[\Phi]\in\mathcal{H}^2(\mathbbm{P}^{s,\eta})$.

Concerning the second statement, \eqref{EbracketGammater} yields
 $$ \underset{t\in[s,T]}{\text{sup }}(\Phi^{s,\eta}_t)^2\leq 4(\Phi^2_s(\eta)+(V_T-V_s)\int_s^TA(\Phi)^2_r dV_r+\underset{t\in[s,T]}{\text{sup }}(M^{s,\eta}[\Phi]_t)^2) \in \mathcal{L}^1,$$ by Doob's inequality because
 $M^{s,\eta}[\Phi]\in\mathcal{H}^2(\mathbbm{P}^{s,\eta})$.
\end{proof}

We will now be interested in a specific type of driving MAF $N$, see Definition \ref{DrivingMAF}.

We fix $\Psi^1,\cdots,\Psi^d\in\mathcal{D}(A)$ verifying the following.
\begin{hypothesis}\label{HypN2}
For every $1 \le i\leq d$,
\begin{enumerate}
\item $\Psi^i,A(\Psi^i)\in\mathcal{L}^2_{uni}$;
\item $(\Psi^i)^2\in\mathcal{D}(A)$;
\item $\Gamma(\Psi^i)$ is bounded.
\end{enumerate}
\end{hypothesis}

\begin{remark}\label{RN2}
\begin{enumerate}\
\item 
$(M[\Psi^1],\cdots,M[\Psi^d])$ is a vector of square integrable MAFs verifying Hypothesis \ref{HypN}. This follows because of  Lemmas \ref{bracketGammabis}
and \ref{bracketGammater}.
\item
In the sequel of Section \ref{S2}, we therefore work with the driving
 MAF 
$N = M[\Psi]:=(M[\Psi^1],\cdots,M[\Psi^d])$, see
Definition \ref{DrivingMAF}.
With this choice we fit the framework of
Section \ref{S2Prelim}. 
\item In particular,   Theorem \ref{T45}
and Proposition \ref{Defuv} apply:
 for all $(s,\eta)$ there is a unique solution $(Y^{s,\eta},M^{s,\eta})$ of 
the $BSDE^{s,\eta}(f,\xi)$ \eqref{ET45},
 where the driving MAF is $N=M[\Psi]$. 
\end{enumerate}
\end{remark}

We now consider the following abstract path-dependent non linear equation.
\begin{equation}\label{AbstractEq}
\left\{
\begin{array}{l}
A\Phi + f(\cdot,\cdot,\Phi,\Gamma(\Phi,\Psi))=0\text{ on }[0,T]\times\Omega\\
\Phi_T=\xi\text{ on }\Omega.
\end{array}\right.
\end{equation}

Inspired by an analogous notion in the Markovian framework (see \cite{paper3}) 
and by the classical notion of mild solution, we introduce the
 corresponding notion of {\it decoupled mild solution} for a path-dependent evolution equation.
\begin{definition}\label{Abmildsoluv}
	Let $Y,Z^1,\cdots,Z^d\in\mathcal{L}^2_{uni}$ (see Notation \ref{L2uni}) and denote $Z:=(Z_1,\cdots,Z_d)$.
	\begin{enumerate}
		\item 
		
		The couple $(Y,Z)$ will be called {\bf solution
			of the identification problem related to $(f,\xi)$}
		or simply {\bf solution of} 
		$IP(f,\xi)$ if  for every $(s,\eta)\in[0,T]\times \Omega$,
		\small
		\begin{equation}\label{AbMildEq}
		\left\{
		\begin{array}{rl}
		Y_s(\eta)&=P_s[\xi](\eta)+\int_s^TP_s\left[f\left(r,\cdot,Y_r,Z_r\right)\right](\eta)dV_r\\
		(Y\Psi^1)_s(\eta) &=P_s[\xi \Psi^1_T](\eta) -\int_s^TP_s\left[\left(Z^1_r+Y_rA(\Psi^1)_r-\Psi^1_rf\left(r,\cdot,Y_r,Z_r\right)\right)\right](\eta)dV_r\\
		&\cdots\\
		(Y\Psi^d)_s(\eta) &=P_s[\xi \Psi^d_T](\eta) -\int_s^TP_s\left[\left(Z^d_r+Y_rA(\Psi^d)_r-\Psi^d_rf\left(r,\cdot,Y_r,Z_r\right)\right)\right](\eta)dV_r.
		\end{array}\right.
		\end{equation}
		\normalsize
		\item A process $Y$ will be called  {\bf decoupled mild solution}
		of \eqref{AbstractEq} if there exist some $Z$
		such that the couple $(Y,Z)$ is a solution 
		of $IP(f,\xi)$.
	\end{enumerate}
\end{definition}

\begin{theorem}\label{AbstractTheorem}
Assume the validity of Hypothesis \ref{HypN2} 
for $\Psi^1, \ldots, \Psi^d$ and of
Hypothesis  \ref{HypBSDE} for $(f, \xi)$.
\begin{enumerate}
\item $IP(f,\xi)$ admits a unique solution
 $(Y,Z)\in\mathcal{L}^2_{uni}\times (L^2_{uni})^d$.
By uniqueness we mean more precisely the following: 
if $(Y,Z)$ and $(\bar Y, \bar Z)$ are two solutions then
$Y $ and $\bar Y$ are identical and $Z = \bar Z$ q.s.
In particular, there is a unique decoupled mild solution $Y$
of \eqref{AbstractEq}.
\item For every $(s,\eta)$, let 
  $(Y^{s,\eta},M^{s,\eta})$ be the solution of $BSDE^{s,\eta}(f,\xi)$ \eqref{ET45} with $N^{s,\eta}=M[\Psi]^{s,\eta}$. Then, for every $(s,\eta)$,
$Y_s(\eta)= Y^{s,\eta}_s$, and $(Z^1, \ldots, Z^d)$ are identified
as in item 3. of Proposition \ref{Defuv} with
 $N = M[\Psi]$.
\end{enumerate}
\end{theorem}


\begin{proof} 
	We start establishing existence in item 1.
 Let $Y,Z$ be the processes introduced in Proposition \ref{Defuv}. A direct consequence of that proposition, of Lemma \ref{ModifImpliesdV} and of equations $BSDE^{s,\eta}(f,g)$ is that for every $(s,\eta)\in[0,T]\times\Omega$, we have 
	\begin{equation}
	Y_s=\xi+\int_s^Tf\left(r,\cdot,Y_r,Z_r\right)dV_r -(M^{s,\eta}_T-M^{s,\eta}_s)\quad \mathbbm{P}^{s,\eta}\text{ a.s.}
	\end{equation}
 Taking the expectation, applying Fubini's theorem and using the fact that
 $M^{s,\eta}$ is a martingale, in agreement with Remark \ref{Borel}
 we get
\begin{equation} \label{E412}
Y_s(\eta)=P_s[\xi](\eta)+\int_s^TP_s[f(r,\cdot,Y_r,Z_r)](\eta)dV_r.
\end{equation}	
We now fix an integer $1 \le i \leq d$ and $(s,\eta)\in[0,T]\times\Omega$. 
By Remark \ref{RMafPhi}. we recall that the process $\Psi^i$ admits on $[s,T]$ a $\mathbbm{P}^{s,\eta}$-modification which is a cadlag special semimartingale with decomposition
\begin{equation} \label{E412bis}
	\Psi^{i,s,\eta}=\Psi^i_s(\eta)+\int_s^{\cdot}A(\Psi^i)_rdV_r+M[\Psi^i]^{s,\eta}.
\end{equation}
Applying the integration by parts formula to $Y^{s,\eta}\Psi^{i,s,\eta}$, we get
\begin{eqnarray*}
	d(Y^{s,\eta}\Psi^{i,s,\eta})_t&=&-\Psi^{i,s,\eta}_tf(t,\cdot,Y^{s,\eta},Z^{s,\eta})dV_t+\Psi^{i,s,\eta}_{t^-}dM^{s,\eta}_t+Y^{s,\eta}A(\Psi^i)_tdV_t\\
&+&Y^{s,\eta}_tdM[\Psi^i]^{s,\eta}_t+d[M^{s,\eta},M[\Psi^i]^{s,\eta}]_t,
\end{eqnarray*}
hence integrating between $s$ and $T$, by Proposition \ref{Defuv} and by Lemma \ref{ModifImpliesdV},
\begin{eqnarray}\label{eqYZ}
	Y_s(\eta)\Psi^i_s(\eta)&=&\xi\Psi^i_T-\int_s^T(Y_rA(\Psi^i)_r-\Psi^i_rf(r,\cdot,Y_r,Z_r))dV_r-\int_s^T\Psi^{i,s,\eta}_{r^-}dM^{s,\eta}_r  \nonumber \\
&& \\
&-&\int_s^TY^{s,\eta}_{r^-}dM[\Psi^i]^{s,\eta}_r-[M^{s,\eta},M[\Psi^i]^{s,\eta}]_T. \nonumber
\end{eqnarray}
We wish once again to take the expectation and to use Fubini's theorem. 
 Since $Y,A(\Psi^i),\Psi^i,f(\cdot,\cdot,Y,Z)\in\mathcal{L}^2_{uni}$ then $(Y_rA(\Psi^i)_r-\Psi^i_rf(r,\cdot,Y_r,Z_r))\in\mathcal{L}^1_{uni}$. By Lemma \ref{bracketGammater} we have  
 $\underset{t\in[s,T]}{\text{ sup }}|\Psi^{i,s,\eta}_t|\in\mathcal{L}^2(\mathbbm{P}^{s,\eta})$  and $M[\Psi^i]^{s,\eta}\in\mathcal{H}^2(\mathbbm{P}^{s,\eta})$ and we recall that this implies that 
  $\int_s^{\cdot}\Psi^{i,s,\eta}_{r^-}dM^{s,\eta}_r$ is a martingale, see for example  Lemma 3.15 in \cite{paper1preprint}.
 $\int_s^{\cdot}Y^{s,\eta}_{r^-}dM[\Psi^i]^{s,\eta}_r$ is also a martingale,
see Remark A.11 in \cite{paper3}.
Finally, since by Lemma \ref{bracketGammater}, $\Psi^i_T\in \mathcal{L}^2(\mathbbm{P}^{s,\eta})$  and since 
 $\xi\in \mathcal{L}^2(\mathbbm{P}^{s,\eta})$, then $\xi\Psi^i_T\in \mathcal{L}^1(\mathbbm{P}^{s,\eta})$. So we can take the expectation in \eqref{eqYZ} to get
 \small
\begin{equation}
	\begin{array}{rl}
		Y_s(\eta)\Psi^i_s(\eta)&=\mathbbm{E}^{s,\eta}\left[\xi\Psi^i_T-\int_s^T(Y_rA(\Psi^i)_r-\Psi^i_rf(r,\cdot,Y_r,Z_r))dV_r-[M^{s,\eta},M[\Psi^i]^{s,\eta}]_T\right]\\&=\mathbbm{E}^{s,\eta}\left[\xi\Psi^i_T-\int_s^T(Y_rA(\Psi^i)_r-\Psi^i_rf(r,\cdot,Y_r,Z_r))dV_r-\langle M^{s,\eta},M[\Psi^i]^{s,\eta}\rangle_T\right]\\
		&=\mathbbm{E}^{s,\eta}\left[\xi\Psi^i_T-\int_s^T(Z^i_r+Y_rA(\Psi^i)_r-\Psi^i_rf(r,\cdot,Y_r,Z_r))dV_r\right],
	\end{array}
\end{equation} 
\normalsize
where the latter equality yields from Proposition \ref{Defuv}. Since $Z^i\in\mathcal{L}^2_{uni}$ and $(Y_rA(\Psi^i)_r-\Psi^i_rf(r,\cdot,Y_r,Z_r))\in\mathcal{L}^1_{uni}$, by Fubini's Theorem we have 

\begin{equation}
Y_s(\eta)\Psi^i_s(\eta)=P_s[\xi\Psi^i_T](\eta)-\int_s^TP_s[Z^i_r+Y_rA(\Psi^i)_r-\Psi^i_rf(r,\cdot,Y_r,Z_r)](\eta)dV_r.
\end{equation}
This shows existence in item 1. 
 The validity of item 2. comes from the choice of $(Y,Z)$ 
and by Proposition \ref{Defuv}.

We will now proceed showing uniqueness in item 1.
We assume the existence of $U,W^1,\cdots,W^d\in\mathcal{L}^2_{uni}$ such that 
for all $(s,\eta)\in[0,T]\times \Omega$, 
\begin{equation}\label{eqUW}
\begin{array}{rl}
U_s(\eta)&=P_s[\xi](\eta)+\int_s^TP_s[f(r,\cdot,U_r,W_r)](\eta)dV_r\\
U_s(\eta)\Psi^i_s(\eta)&=P_s[\xi\Psi^i_T](\eta)-\int_s^TP_s[W^i_r+U_rA(\Psi^i)_r-\Psi^i_rf(r,\cdot,U_r,W_r)](\eta)dV_r\\
&1\leq i\leq d.
\end{array}
\end{equation}
 We will show that $U = Y$ and that for all $(s,\eta)$,   $1 \le i \le d$,
 \\
$ W^i =  \frac{d\langle M^{s,\eta},M[\Psi^i]^{s,\eta}\rangle_r}{dV_r}, \quad dV \otimes d\mathbbm{P}^{s,\eta} {\rm a.e.}$
hence that $W=Z$ q.s.
 
We fix $(s,\eta)\in[0,T]\times\Omega$.
 We define the process $\bar M$ as being
 equal to $0$ on $[0,s]$ and to 
$U_t-U_s(\eta)+\int_s^tf(r,\cdot,U_r,W_r)dV_r$ $t\in [s,T]$. Let us fix
  $t\in [s,T]$. We emphasize that $U$ and $\bar M$ are a priori not cadlag.
 Applying the first line of \eqref{eqUW} to $(s,\eta):=(t,\omega)$, we get, $\mathbbm{P}^{s,\eta}$ a.s.
\begin{equation} \label{E321}
\begin{array}{rcl}
U_t(\omega)&=&P_t[\xi](\omega)+\int_t^TP_t[f(r,\cdot,U_r,W_r)](\omega)dV_r\\
&=&\mathbbm{E}^{t,\omega}\left[\xi+\int_t^Tf(r,\cdot,U_r,W_r)dV_r\right]\\
&=&\mathbbm{E}^{s,\eta}\left[\xi+\int_t^Tf(r,\cdot,U_r,W_r)dV_r|\mathcal{F}_t\right],
\end{array}
\end{equation}
thanks to  Remark \ref{Borel}.
 From this we deduce that for all  $ t\in[s,T]$,
\begin{equation}\label{E416}
 \bar M_t=\mathbbm{E}^{s,\eta}\left[\xi+\int_s^Tf(r,\cdot,U_r,W_r)dV_r\middle|\mathcal{F}_t\right] -U_s(\eta)  \quad \mathbbm{P}^{s,\eta} {\rm a.s.}
\end{equation}
In particular $\bar{M}$ is a  $\mathbbm{P}^{s,\eta}$-martingale on $[s,T]$.
 Since $\left(\Omega,\mathcal{F}^{s,\eta},\mathbbm{F}^{s,\eta},\mathbbm{P}^{s,\eta}\right)$  fulfills the usual conditions, $\bar{M}$ admits a cadlag version which we will denote $\bar{M}^{s,\eta}$. Then $U$ admits on $[s,T]$ a cadlag version $U^{s,\eta}$ which is a special semimartingale with decomposition
\begin{eqnarray} \label{E419}
	U^{s,\eta} &:=& U_s(\eta)-\int_s^{\cdot}f(r,\cdot,U_r,W_r)dV_r+\bar{M}^{s,\eta} \nonumber \\
&& \\
&=& U_s(\eta)-\int_s^{\cdot}f(r,\cdot,U^{s,\eta}_r,W_r)dV_r+\bar{M}^{s,\eta}, \nonumber
\end{eqnarray}
where the second equality follows by  the fact that $U^{s,\eta}$ is a version of  $U$  
and by Lemma \ref{ModifImpliesdV}.
By  \eqref{E321}, 
 we have $ U^{s,\eta}_T = U_T = \xi$,   $\mathbbm{P}^{s,\eta}$-a.s.
so, 
\begin{equation}\label{eqUWBSDE1}
U^{s,\eta} = \xi + \int_{\cdot}^{T}f(r,\cdot,U^{s,\eta}_r,W_r)dV_r - (\bar{M}^{s,\eta}_T - \bar{M}^{s,\eta}).
\end{equation}
We  show below that $(U^{s,\eta},\bar{M}^{s,\eta})$ solves $BSDE^{s,\eta}(f,\xi)$ on $[s,T]$. For this, we are left to show that $W=\frac{d\langle \bar{M}^{s,\eta},M[\Psi]^{s,\eta}\rangle}{dV}$, $dV\otimes d\mathbbm{P}^{s,\eta}$ a.e., and that $(U^{s,\eta},\bar{M}^{s,\eta})\in\mathcal{L}^2(dV\otimes\mathbbm{P}^{s,\eta})\times\mathcal{H}^2_0(\mathbbm{P}^{s,\eta})$. By  \eqref{E416} together with  Jensen's and Cauchy-Schwarz inequalities we have
\begin{equation}
	\begin{array}{rcl}
		&&\mathbbm{E}^{s,\eta}[\bar M_T^2]\\ 
		&=&\mathbbm{E}^{s,\eta}\left[(\xi-U_s(\eta)+\int_s^Tf(r,\cdot,U_r,W_r)dV_r)^2\right]\\
		&\leq& 4\left(U_s(\eta)^2+\mathbbm{E}^{s,\eta}[\xi^2]+ (V_T-V_s) \mathbbm{E}^{s,\eta}\left[\int_s^Tf^2(r,\cdot,U_r,W_r)dV_r)\right]\right),
	\end{array} 
\end{equation}
where the latter term is finite because $f(\cdot,\cdot,0,0),U,W^1,\cdots,W^d\in\mathcal{L}^2_{uni}$ and because of the Lipschitz condition on $f$. 
So the cadlag version   $\bar{M}^{s,\eta}$ of $\bar M$, belongs to $\mathcal{H}^2_0(\mathbbm{P}^{s,\eta})$.
We fix again  some $1\leq i \leq d$. 

Combining the second line   of \eqref{eqUW} with this fixed integer $i$, applied with $(t,\omega)$ instead of $(s,\eta)$, Fubini's lemma and  Remark  \ref{Borel}, we get the $\mathbbm{P}^{s,\eta}$-a.s. equalities 
\begin{equation} \label{E522}
\begin{array}{rl}
&U_t(\omega)\Psi^i_t(\omega)\\
 =& P_t[\xi\Psi^i_T](\omega) - \int_t^TP_t\left[\left(W^i_r+U_rA(\Psi^i)_r-\Psi^i_rf\left(r,\cdot,U_r,W_r\right)\right)\right](\omega)dV_r\\
=& \mathbbm{E}^{t,\omega}\left[\xi\Psi^i_T-\int_t^T\left(W^i_r+U_rA(\Psi^i)_r-\Psi^i_rf\left(r,\cdot,U_r,W_r\right)\right)dV_r\right] \\
=& \mathbbm{E}^{s,\eta}\left[\xi\Psi^i_T-\int_t^T\left(W^i_r+U_rA(\Psi^i)_r-\Psi^i_rf\left(r,\cdot,U_r,W_r\right)\right)dV_r|\mathcal{F}_t\right](\omega).
\end{array}
\end{equation}
We introduce the process $\tilde M^i$ equal to $0$ on $[0,s[$ and defined on $[s,T]$ by 
\begin{equation}\label{E523bis}
\tilde M^i_t:=U_t\Psi^i_t-U_s(\eta)\Psi^i_s(\eta)-\int_s^t\left(W^i_r+U_rA(\Psi^i)_r-\Psi^i_rf\left(r,\cdot,U_r,W_r\right)\right)dV_r.
\end{equation}
Similarly as for \eqref{E416}, 
we deduce from \eqref{E522} that for all $t \in [s,T]$,  $\mathbbm{P}^{s,\eta}$ a.s.,
$\tilde M^i_t=\mathbbm{E}^{s,\eta}\left[\xi\Psi^i_T-\int_s^T \left(W^i_r+U_rA(\Psi^i)_r-\Psi^i_rf\left(r,\cdot,U_r,W_r\right)\right)dV_r
|\mathcal{F}_t\right]-U_s(\eta)\Psi^i_s(\eta)$
So $\tilde M^i$  is a  $\mathbbm{P}^{s,\eta}$-martingale.
Under $\mathbbm{P}^{s,\eta}$, we consider the cadlag version $\tilde M^{i,s,x}$ of $\tilde M^i$. 
\\
It follows by \eqref{E523bis} that  the process 
\begin{equation} \label{EV1}	
U_s(\eta)\Psi^i_s(\eta)+\int_s^{\cdot}\left(W^i_r+U_rA(\Psi^i)_r-\Psi^i_rf\left(r,\cdot,U_r,W_r\right)\right)dV_r + \tilde M^{i,s,\eta},
\end{equation}
is  a cadlag special semimartingale which is a $\mathbbm{P}^{s,\eta}$-version of $U\Psi^i$ on $[s,T]$ hence indistinguishable from $U^{s,\eta}\Psi^{i,s,\eta}$ on $[s,T]$ since $ U^{s,\eta}$ (resp. $\Psi^{i,s,\eta}$)   is a cadlag version of $U$ (resp. $\Psi^{i}$).
Using  (\ref{E419}) and
integrating by parts, on
 $[s,T]$ we also get for $U^{s,\eta}\Psi^{i,s,\eta}$ the decomposition
\begin{equation} \label{E523}
\begin{array}{rl}
& U^{s,\eta} \Psi^{i,s,\eta} \\
=& 
U_s(\eta)\Psi^i_s(\eta)+\int_s^{\cdot}U^{s,\eta}_rA(\Psi^i)_rdV_r+\int_s^{\cdot}U^{s,\eta}_{r^-}dM[\Psi^i]^{s,\eta}_r\\
&-\int_s^{\cdot}\Psi^{i,s,\eta}_rf(r,\cdot,U_r,W_r)dV_r+\int_s^{\cdot}\Psi^{i,s,\eta}_{r^-}d\bar{M}^{s,\eta}_r+[\bar{M}^{s,\eta},M[\Psi^i]^{s,\eta}]\\
=& U_s(\eta)\Psi^i_s(\eta)+\int_s^{\cdot}(U^{s,\eta}_rA(\Psi^i)_r-\Psi^{i,s,\eta}_rf(r,\cdot,U_r,W_r))dV_r +\langle \bar{M}^{s,\eta},M[\Psi^i]^{s,\eta}\rangle\\
&+\int_s^{\cdot}U^{s,\eta}_{r^-}dM[\Psi^i]^{s,\eta}_r+\int_s^{\cdot}\Psi^{i,s,\eta}_{r^-}d\bar{M}^{s,\eta}_r+\left([\bar{M}^{s,\eta},M[\Psi^i]^{s,\eta}]-\langle \bar{M}^{s,\eta},M[\Psi^i]^{s,\eta}\rangle\right).
\end{array}
\end{equation}
\eqref{EV1} and \eqref{E523} provide now two ways of decomposing the special semi-martingale $U^{s,\eta}\Psi^{i,s,\eta}$ into the sum of an initial value, a bounded variation predictable process vanishing at time $s$ and of a local martingale vanishing at time $s$.

%
By uniqueness of the decomposition of a special semimartingale, 
identifying the bounded variation predictable components
and using Lemma \ref{ModifImpliesdV}
we get 
\begin{equation}
\begin{array}{rcl}
&&\int_s^{\cdot}  (W^i_r+ U_rA(\Psi^i)_r-\Psi^i_rf(r,\cdot,U_r,W_r))dV_r \\
&=&	\langle \bar{M}^{s,\eta},M[\Psi^i]^{s,\eta}\rangle +\int_s^{\cdot}(U_rA(\Psi^i)_r-\Psi^i_rf(r,\cdot,U_r,W_r))dV_r.
\end{array}
\end{equation}
This  yields that $\langle \bar{M}^{s,\eta},M[\Psi^i]^{s,\eta}\rangle$ and $\int_s^{\cdot}W^i_rdV_r$ are indistinguishable on $[s,T]$.  Since this holds for all $i$, 
 thanks to \eqref{eqUWBSDE1} we have
\begin{equation} 
 U^{s,\eta} = \xi+\int_{\cdot}^{T}f\left(r,\cdot,U^{s,\eta}_r,\frac{d\langle \bar{M}^{s,\eta},M[\Psi]^{s,\eta}\rangle_r}
{dV_r}\right)dV_r-(\bar{M}^{s,\eta}_T - \bar{M}^{s,\eta}),
\end{equation}
with $\bar{M}^{s,\eta}\in\mathcal{H}^2_0(\mathbbm{P}^{s,\eta})$. 
This implies of course that $(U^{s,\eta}, \bar{M}^{s,\eta})$ is a solution of $BSDE^{s,\eta}(f,\xi)$.
Thanks to the uniqueness statement for BSDEs, see Theorem \ref{T45},
this shows
$U^{s,\eta} = Y^{s,\eta}$
and $\bar{M}^{s,\eta} = M^{s,\eta}$ in the sense of indistinguishability.
In particular, the first equality gives
\begin{equation} \label{E427bis} 
U_s(\eta)=U^{s,\eta}_s=Y^{s,\eta}_s=Y_s(\eta) \ {\rm a.s.}
 \end{equation}
Since this holds for all $(s,\eta)$, we have $U=Y$. On the other hand, 
for all $i$ we have $W^i_t=\frac{d\langle \bar{M}^{s,\eta},M[\Psi^i]^{s,\eta}\rangle_t}{dV_t} = \frac{d\langle M^{s,\eta},M[\Psi^i]^{s,\eta}\rangle_t}{dV_t}=Z^{i}_t $ 
$dV \otimes dP^{s,\eta}$ a.e. for all $(s,\eta)$ hence $W^i=Z^i$ q.s.
This concludes the proof of uniqueness.
\end{proof}

\section{Decoupled mild solutions of path-dependent PDEs
 and IPDEs}\label{S3}

In this section we keep using Notation \ref{canonicalspace}, but $E=\mathbbm{R}^d$ for some $d\in\mathbbm{N}^*$ and $(X^1_t,\cdots,X^d_t)_{t\in\mathbbm{R}_+}$ will denote the coordinates of the canonical process,see Notation \ref{canonicalspace}. $T > 0$ will be a fixed horizon. 

For the convenience of the reader, the stopped canonical process $(X^1_{\cdot\wedge T},\cdots,X^d_{\cdot\wedge T})$ will still be by denoted $(X^1,\cdots,X^d)$.

\subsection{Path-dependent SDEs}\label{S1c}

We now recall some notions and results concerning a family of path-dependent SDEs with jumps whose solution provide examples of path-dependent canonical classes. In this subsection, all results come from Section 5 in \cite{paperMPv2}. We will also refer to notions of \cite{jacod} Chapters II, III, VI and \cite{jacod79} Chapter XIV.5.

\begin{notation}\label{Omegat}
	For any $t\in\mathbbm{R}_+$, we denote $\Omega^t:=\{\omega\in\Omega:\omega=\omega^t\}$ the set of paths which are constant after time $t$.
	We also denote $\Lambda:=\{(s,\eta)\in\mathbbm{R}_+\times\Omega: \eta\in\Omega^s\}$.
\end{notation}

\begin{proposition}\label{BorelLambda}
	$\Lambda$ is a closed subspace of $\mathbbm{R}_+\times\Omega$, hence a Polish space when equipped with the induced topology. 
	The Borel $\sigma$-field $\mathcal{B}(\Lambda)$ is equal to the trace $\sigma$-field $\Lambda\cap\mathcal{P}ro^o$.
\end{proposition}
From now on, $\Lambda$, introduced in Notation \ref{Omegat}, is equipped with the induced topology and the trace $\sigma$-field. 
\\
\\
We fix a bounded positive measure $F$ on $(\mathbbm{R}^d,\mathcal{B}(\mathbbm{R}^d))$ not charging $0$, and some coefficients:
\begin{itemize}
	\item $\beta$, a bounded $\mathbbm{R}^d$-valued $\mathbbm{F}^o$-predictable process;
	\item $\sigma$, a bounded $\mathbbm{M}_d(\mathbbm{R})$-valued $\mathbbm{F}^o$-predictable process;
	\item $\gamma$, a bounded $\mathbbm{R}^d$-valued  $\mathcal{P}re^o\otimes\mathcal{B}(\mathbbm{R}^d)$-measurable function on $\mathbbm{R}_+\times\Omega\times\mathbbm{R}^d$,
\end{itemize}
defined on the canonical space.

\begin{definition} \label{D213}
	Let $(s,\eta)\in\mathbbm{R}_+\times\Omega$.
	We call \textbf{weak solution of the SDE with coefficients $\beta$, $\sigma$, $\gamma$ and starting in $(s,\eta)$} any probability measure $\mathbbm{P}^{s,\eta}$ on $(\Omega,\mathcal{F})$ such that there exists a  stochastic basis fulfilling the usual conditions $(\tilde \Omega,\tilde{\mathcal{F}},\tilde{\mathbbm{F}},\mathbbm{\tilde P})$ on which is defined a $d$-dimensional
 Brownian motion $W$ and a Poisson measure $p$ of intensity 
	$q(dt,dx):=dt\otimes F(dx)$,  $W,p$ being optional for the
 filtration $\tilde{\mathbbm{F}}$, 
 a $d$-dimensional   $\tilde{\mathbbm{F}}$-adapted cadlag process $\tilde X$ such that $\mathbbm{P}^{s,\eta}=\tilde{\mathbbm{P}}\circ \tilde X^{-1}$ and such that the following holds.
	\\
	\\
	Let $\tilde \beta$, $\tilde \sigma$, $\tilde \gamma$ be defined by   $\tilde \beta_t:=\beta_t\circ\tilde X$, $\tilde \sigma_t:=\sigma_t\circ\tilde X$ for all $t\in\mathbbm{R}_+$ and $\tilde \gamma_t(\cdot,x):=\gamma_t(\tilde X,x)$ for all $(t,x)\in\mathbbm{R}_+\times\mathbbm{R}^d$. Then 
	\begin{itemize}
		\item for all $t\in[0,s]$, $\tilde X_t=\eta(t)$
		$\mathbbm{\tilde P}$ a.s.;
		\item $\tilde X_t=\eta(s)+\int_s^t\tilde \beta_rdr+\int_s^t\tilde \sigma_rdW_r+\tilde \gamma\star(p-q)_t$ $\mathbbm{\tilde P}$ a.s. for all $t\geq s$,
	\end{itemize}
	where $\star$ is the integration against random measures, see \cite{jacod} Chapter II.2.d for instance.
\end{definition}


\begin{definition}\label{Char}
	Let $s\in\mathbbm{R}_+$ and $(Y_t)_{t\geq s}$ be a cadlag special semimartingale defined on the canonical space with (unique) decomposition $Y=Y_s+B+M^c+M^d$,
 where $B$ is predictable with bounded variation, $M^c$ a continuous local martingale, $M^d$ a purely discontinuous martingale, all three vanishing at time $s$. We will call \textbf{characteristics} of $Y$ the triplet $(B,C,\nu)$ where $C=\langle M^c\rangle$ and $\nu$ is the predictable compensator of the measure of the jumps of $Y$.
\end{definition}

There are several  equivalent characterizations of weak solutions of path-dependent SDEs with jumps which we will now state in our setup.

\begin{notation}\label{NotA}
	For any  $f\in\mathcal{C}_b^2(\mathbbm{R}^d)$ and $t\geq 0$, we denote by $A_tf$ the r.v.
	\begin{equation}
	\begin{array}{rl}
	&\beta_t\cdot \nabla f(X_t)+\frac{1}{2}Tr(\sigma\sigma^{\intercal}\nabla^2f(X_t))\\
	+&\int_{\mathbbm{R}^d}(f(X_t+\gamma_t(\cdot,y))-f(X_t)-\nabla f(X_t)\cdot \gamma_t(\cdot,y))F(dy).
	\end{array}
	\end{equation}
\end{notation}

\begin{proposition}\label{SDEeq}
	Let $(s,\eta)\in\mathbbm{R}_+\times\Omega$ be fixed and let $\mathbbm{P}\in\mathcal{P}(\Omega)$. There is equivalence between the  following items.
	\begin{enumerate}
		\item $\mathbbm{P}$ is a weak solution of the SDE with coefficients $\beta,\sigma,\gamma$ starting in $(s,\eta)$.
		\item $\mathbbm{P}(\omega^s=\eta^s)=1$ and, under $\mathbbm{P}$,
		$(X_t)_{t\geq s}$ is a special semimartingale with characteristics 
		\begin{equation}
		\begin{array}{lcl}
		B&=&\int_s^{\cdot}\beta_rdr;
		\\
		C&=&\int_s^{\cdot}(\sigma\sigma^{\intercal})_rdr;\\
		\nu&:&(\omega,A)\longmapsto\int_s^{+\infty}\int_{\mathbbm{R}^d}\ \mathds{1}_A(r,\gamma_r(\omega,y))\mathds{1}_{\{\gamma_r(\omega,y)\neq 0\}}F(dy)dr.
		\end{array}
		\end{equation}
		\item  $\mathbbm{P}(\omega^s=\eta^s)=1$
		and  $f(X_{\cdot})-\int_s^{\cdot}A_rfdr$ is on $[s,+\infty[$ a $(\mathbbm{P},\mathbbm{F})$-martingale for all 
		$f\in\mathcal{C}_b^2(\mathbbm{R}^d)$.
	\end{enumerate}
\end{proposition}

\begin{theorem}\label{SDEcond}
	Assume  that for every $(s,\eta)\in\mathbbm{R}_+\times\Omega$, the SDE with coefficients $\beta$, $\sigma$, $\gamma$ and starting in $(s,\eta)$ admits a unique weak solution $\mathbbm{P}^{s,\eta}$. Then $(\mathbbm{P}^{s,\eta})_{(s,\eta)\in\mathbbm{R}_+\times \Omega}$ is a path-dependent canonical class verifying Hypothesis \ref{HypClass}.
\end{theorem}

We introduce the following hypothesis on the coefficients $\beta,\sigma,\gamma$.
\begin{hypothesis}\label{HypWellPosed}
	\begin{enumerate}\
		\item $\beta,\sigma$ (resp. $\gamma$) are bounded and  for Lebesgue almost all $t$ (resp. $dt\otimes dF$ almost all $(t,y)$), $\beta_t,\sigma_t$ (resp. $\gamma_t(\cdot,y)$) are continuous.
		\item For every $(s,\eta)\in\mathbbm{R}_+\times\Omega$ there exists a unique weak solution $\mathbbm{P}^{s,\eta}$ of the SDE of coefficients $\beta,\sigma,\gamma$ starting in $(s,\eta)$, see Definition \ref{D213}.
	\end{enumerate}
\end{hypothesis}

We recall two classical examples of conditions on the coefficients for which it is known that there is existence
 and uniqueness of a weak solution for the path-dependent SDE,
see Theorem 14.95 and Corollary 14.82 in \cite{jacod79}.
\begin{example} \label{E59}
 Assume that
 $\beta,\sigma,\gamma$ are bounded. Moreover we suppose that
  for all $n\in\mathbbm{N}^*$ there exist
$K^n_2\in L^1_{loc}(\mathbbm{R}_+)$ and
  a Borel function
$K^n_3: \mathbbm{R}^m \times \mathbbm{R}_+ \rightarrow \mathbbm{R} $ 
 such that $\int_{\mathbbm{R}^m} K^n_3(\cdot,y)F(dy)\in L^1_{loc}(\mathbbm{R}_+)$ verifying the following.
For all $x\in\mathbbm{R}^m$, $t\geq 0$ and $\omega,\omega'\in\Omega$ such that $\underset{r\leq t}{\text{sup }}\|\omega(r)\|\leq n$ and $\underset{r\leq t}{\text{sup }}\|\omega'(r)\|\leq n$, we have 
\begin{itemize}
	\item  $\|\sigma_t(\omega)-\sigma_t(\omega')\|\leq K^n_2(t)\underset{r\leq t}{\text{sup }}\|\omega(r)-\omega'(r)\|^2$;
	\item 
$\|\gamma(t,\omega,x)-\gamma(t,\omega',x)\|\leq K^n_3(t,x)\underset{r\leq t}{\text{sup }}\|\omega(r)-\omega'(r)\|^2$.
\end{itemize}
Finally assume that one of the two following hypotheses are fulfilled.
\begin{enumerate}
\item There exists $K^n_1 \in L^1_{loc}(\mathbbm{R}_+)$ such that 
for 
all $t\geq 0$ and $\omega\in\Omega$, 
 $\|\beta_t(\omega)-\beta_t(\omega')\|\leq K^n_1(t)\underset{r\leq t}{\text{sup }}\|\omega(r)-\omega'(r)\|$;
\item  there exists $c>0$ such that for 
all $x\in\mathbbm{R}^m$, $t\geq 0$ and $\omega\in\Omega$, $x^{\intercal}\sigma_t(\omega)\sigma_t(\omega)^{\intercal}x\geq c\|x\|^2$.
\end{enumerate}
Then item 2. of Hypothesis  \ref{HypWellPosed} is satisfied.
\end{example}
\begin{proposition}\label{UniqueMPcontImpliesProg}
	Assume that Hypothesis \ref{HypWellPosed} holds.
	Then $(\mathbbm{P}^{s,\eta})_{(s,\eta)\in\mathbbm{R}_+\times \Omega}$ is a progressive path-dependent canonical class verifying Hypothesis \ref{HypClass}.

\end{proposition}

\subsection{Dupire's derivatives and path-dependent stochastic calculus}\label{S3a}

We will recall here some notions and results introduced in \cite{dupire} and later developed in \cite{contfournie13}.

\begin{definition}
	From now on, an $\mathbbm{F}^o$-progressively measurable process with values in $\mathbbm{R}^n$ for some $n\in\mathbbm{N}^*$ will also be called an $\mathbbm{R}^n$-{\bf valued functional}.
If $n=1$, $\Phi$ will be said real valued functional.
\end{definition}
We recall that such an $\mathbbm{R}^n$-{\bf valued functional} can also be seen (by considering its restriction
 on $\Lambda$) as a Borel function
 from $\Lambda$ to $\mathbbm{R}^n$, see Definition \ref{Omegat} and Proposition \ref{BorelLambda}. In the sequel we will not distinguish between an $\mathbbm{F}^o$-progressively measurable process and its restriction to $\Lambda$.
\begin{notation} 
	For all $t\geq 0$, we denote $\Lambda_t:=\{(s,\eta)\in[0,t]\times\Omega:\eta\in\Omega^s\}$ which is clearly a closed subspace of $\Lambda$ and of $\mathbbm{R}_+\times\Omega$.
	On $\Lambda_T$ we denote by $d_{\infty}$ the distance defined by
	 $d_{\infty}((s_1,\eta_1),(s_2,\eta_2)):= \underset{t\in[0,T]}{\text{sup }}|\eta_2(t)-\eta_1(t)|+|s_2-s_1|$.
\end{notation}
This distance induces a topology on  $\Lambda_T$ which is stronger than its natural induced topology inherited from $\mathbbm{R}_+\times \Omega$.
\begin{definition}
	Let  $\Phi$ be some $\mathbbm{R}^n$-valued functional. $\Phi$ will be said to be {\bf continuous} if it is continuous with respect to $d_{\infty}$.

\end{definition}

The following definitions and notations are adapted from \cite{contfournie10}. 
\begin{definition}
	In the whole definition, we fix $\Phi$  a
 real valued functional, 
 constant after time $T$, ie such that $\Phi_t(\omega)=\Phi_{t\wedge T}(\omega)$ for all $(t,\omega)$.
	
Let $(s,\eta)\in\Lambda_T$. We say that $\Phi$ is \textbf{vertically differentiable at} $(s,\eta)$ if 
\begin{equation}
\begin{array}{rcl}
x&\longmapsto& \Phi_s(\eta+ x\mathds{1}_{[s,T]})\\
\mathbbm{R}^d&\longrightarrow& \mathbbm{R},
\end{array}
\end{equation}
is differentiable in $0$. The corresponding gradient at $0$ is denoted $\nabla \Phi_s(\eta)$.
	
	We say that $\Phi$ is \textbf{vertically differentiable} if it is vertically differentiable in $(s,\eta)$ for all $(s,\eta)\in\Lambda_T$. In this case, $\nabla\Phi:(s,\eta)\longmapsto \nabla\Phi(s,\eta)$
defined on $\Lambda_T$, 
will be called the \textbf{gradient} of $\Phi$. 
We remark that, whenever that derivation gradient is Borel,
 it defines an $\mathbbm{R}^d$-valued functional. Its coordinates will be denoted $(\nabla_i \Phi)_{i\leq d}$.

	Similarly, we can define the {\bf Hessian} $\nabla^2 \Phi_s(\eta)$ of $\Phi$ at some point $(t,\eta)$. It belongs to the space of symmetric matrices of size $d$ and its coordinates will be denoted $(\nabla^2_{i,j} \Phi_s(\eta))_{i,j\leq d}$.

	Let $(s,\eta)\in\Lambda_T$ (implying that $\eta$ is constant
 after time $s$). We say that $\Phi$ is \textbf{horizontally differentiable at}
$(s,\eta)\in\Lambda_T$, $s < T$; if
\begin{equation}
\begin{array}{rcl}
t&\longmapsto& \Phi_t(\eta)\\
\ [s,T]&\longrightarrow& \mathbbm{R},
\end{array}
\end{equation}
admits a right-derivative at $s$. The corresponding derivative
 will be denoted $D \Phi_s(\eta)$. 

	We say that $\Phi$ is \textbf{horizontally differentiable} if it is horizontally differentiable in $(s,\eta)$ for all $(s,\eta)\in\Lambda_T$
such that $s < T$ and the limit  $D \Phi_T(\eta):= \lim_{s \uparrow T} 
D \Phi_s(\eta)$ exists for every $\eta \in \Omega^T$.
 In this case, $D\Phi:(s,\eta)\longmapsto D \Phi_s(\eta)$ will be called the \textbf{horizontal derivative} of $\Phi$.
 
  If it is Borel, it defines a real valued functional.	$\Phi$ will be said \textbf{continuous at fixed times} if for all $t\in[0,T]$, $\Phi_t(\cdot):\Omega^t\longmapsto \mathbbm{R}$ is continuous with respect to the sup norm on $\Omega^t$.
	
	By convention, if $\Psi=D\Phi,\nabla \Phi,\nabla^2\Phi$ is well-defined on $\Lambda_T$, it will be extended on $[0,T]\times \Omega$ by setting $\Psi_t(\omega):=\Psi_t(\omega^t)$ and on $]T,+\infty[\times\Omega$ by the value $0$.
	
	$\Phi$ will be said \textbf{left-continuous} if for all $t\in[0,T]$, $\epsilon>0$, $\omega\in\Omega^t$, there exists $\zeta>0$ such that for any $(t',\omega')\in\Lambda_t$ verifying $d_{\infty}((t,\omega),(t',\omega'))<\zeta$ we have $|\Phi_t(\omega)-\Phi_{t'}(\omega')|\leq \epsilon$.
	
	$\Phi$ will be said \textbf{boundedness preserving} if for any compact set $K$ of $\mathbbm{R}^d$ there exists a constant $C_K>0$ such that
 for all $t\in[0,T]$ and $\omega\in\Omega^t$ taking values in $K$, we have $|\Phi_t(\omega)|\leq C_K$.
	
	$\Phi$ will be said to \textbf{have the horizontal local Lipschitz property} if for all $(t,\omega)\in\Lambda_T$, there exists $C>0$ and $\zeta>0$ such that for all $(s,\eta)\in\Lambda_T$ verifying $d_{\infty}((t,\omega),(s,\eta))\leq \zeta$ we have for all $0\leq t_1\leq t_2\leq s$ that $|\Phi_{t_2}(\eta^{t_1})-\Phi_{t_1}(\eta^{t_1})|\leq C|t_2-t_1|$.
\end{definition}

\begin{notation}
%
	We denote by $\mathbbm{C}^{1,2}_b(\Lambda_T)$ the space of real valued functionals $\Phi$ constant after time $T$,
 which admit a horizontal derivative 
 and vertical
 derivatives up to order two such that $\Phi,D\Phi,\nabla \Phi,\nabla^2\Phi$ 
 are boundedness preserving, left-continuous and
 are continuous at fixed time.
\end{notation}
This space is stable by pointwise sum and product.

\begin{notation}
	For every $\omega\in\Omega$, and $t\geq 0$, we denote by $\omega^{t^-}$ the element of $\Omega^t$ defined by $\omega^{t^-}(r)=\omega(r)$ if $r\in[0,t[$ and $\omega^{t^-}(r)=\omega(t^-)$ otherwise. 
	\\
	For any process $Y$ and time $t\geq 0$, we denote $\Delta Y_t:=Y_t-Y_{t^-}$.
\end{notation}

The following path-dependent It\^o formula comes from 
 \cite{contfournie10} Proposition 6.1. We formulate it in our setup.

\begin{theorem}\label{ItoFunct}
	Let $\mathbbm{P}$ be a probability measure on the canonical space $(\Omega,\mathcal{F})$.
	 Let $s\in[0,T]$ and assume that under probability $\mathbbm{P}$, the canonical process $X$ is such that $(X_t)_{t\in[s,T]}$ is an  $(\mathcal{F}_t)_{t\in[s,T]}$-semimartingale. Let $\Phi\in\mathbbm{C}^{1,2}_b(\Lambda_T)$ and assume that  $\nabla \Phi$ has the horizontal local Lipschitz property.
	Then $(\Phi_t)_{t\in[s,T]}$ is an  $(\mathcal{F}_t)_{t\in[s,T]}$-semimartingale and we have
	\begin{eqnarray*}
	\Phi&=&\Phi_s+\int_s^{\cdot}D\Phi_rdr + \int_s^{\cdot}\nabla \Phi_rdX_r + \frac{1}{2}\int_s^{\cdot}Tr((\nabla^2 \Phi_r)^{\intercal}d\langle X^c\rangle_r)\\
 &+&\underset{r\in]s,\cdot]:\Delta X_r\neq 0}{\sum}(\Phi_r(\omega)-\Phi_r(\omega^{r-})-\nabla  \Phi_r(\omega^{r-})\cdot\Delta X_r),
	\end{eqnarray*}
	in the sense of $\mathbbm{P}$-indistinguishability.
\end{theorem}

We recall the following elementary example.
\begin{lemma} \label{L69}
Let $h\in\mathcal{C}^{1,2}([0,T]\times\mathbbm{R}^d)$ then $H:(t,\omega)\longmapsto h(t\wedge T,\omega(t\wedge T))$ belongs to $\mathbbm{C}_b^{1,2}(\Lambda_T)$ and $DH:(t,\omega)\mapsto \partial_th(t,\omega(t))\mathds{1}_{[0,T]}(t)$, $\nabla H:(t,\omega)\mapsto \nabla_x h(t,\omega(t))\mathds{1}_{[0,T]}(t)$, $\nabla^2H:(t,\omega)\mapsto \nabla_x^2 h(t,\omega(t))\mathds{1}_{[0,T]}(t)$. 
\end{lemma}
\begin{remark}\label{R59} 
	\begin{enumerate}\
		\item In Lemma \ref{L69},
 it is moreover clear that if $h$ does not depend on $t$ then $\nabla H$ has the horizontal local Lipschitz property, hence that Theorem \ref{ItoFunct} above applies for $\Phi=H$.
		\item In particular,	for any $i\in[\![1,d]\!]$, $X^i\in\mathbbm{C}^{1,2}_b(\Lambda_T)$, $DX^i\equiv 0$, $\nabla X^i\equiv e_i\mathds{1}_{[0,T]}$ and $\nabla^2 X^i\equiv 0$, where $(e_1,\cdots,e_d)$ denotes the Euclidean basis of $\mathbbm{R}^d$.
	\end{enumerate}
\end{remark}

\subsection{Decoupled mild solutions of Path-dependent (I)PDEs}\label{S3b}

From now on, we suppose $V(t) \equiv t$.
We fix some coefficients  $\beta,\sigma,\gamma$ verifying Hypothesis \ref{HypWellPosed} but vanishing after time $T$. We denote by $(\mathbbm{P}^{s,\eta})_{(s,\eta)\in\mathbbm{R}_+\times \Omega}$ the weak solution 
 of the corresponding SDE. By Proposition \ref{UniqueMPcontImpliesProg}
 it defines a progressive path-dependent canonical class verifying Hypothesis \ref{HypClass}. We denote again by $(P_s)_{s\in\mathbbm{R}_+}$  the associated path-dependent system of projectors, see Definition \ref{ProbaOp}.
\begin{definition}\label{poly}
Let $\Phi$ be an $\mathbbm{F}^o$-progressively measurable process constant
 after time $T$. $\Phi$ will be said to have \textbf{polynomial growth} if
 there exists $C>0$, $p\in\mathbbm{N}^*$ such that for all 
$(t,\omega)\in[0,T]\times\Omega$, $|\Phi_t(\omega)|\leq C(1+\underset{r\leq t}{\text{sup}}\|\omega(r)\|^p)$.

A r.v. $\xi$ will be said to have \textbf{polynomial growth} if there exists $C>0$, $p\in\mathbbm{N}^*$ such that  $|\xi(\omega)|\leq C(1+\underset{r\leq T}{\text{sup}}\|\omega(r)\|^p)$.
\end{definition}

\begin{lemma}\label{Moments}
For any finite 
$p \ge 1$, $(s,\eta)\in[0,T]\times \Omega$, $\underset{t\in[s,T]}{\text{sup }}|X^i_t| \in \mathcal{L}^p(\mathbbm{P}^{s,\eta})$.
\end{lemma}
\begin{proof} 
We fix some $(s,\eta)\in[0,T]\times \Omega$, $1\leq i\leq d$ and some finite $p\geq 1$. A direct consequence of Proposition \ref{SDEeq} item 2. and of Definition \ref{Char} a is that under $\mathbbm{P}^{s,\eta},$
 $X^i$ may be decomposed on $[s,T]$ as $\eta_i(s) + \int_s^{\cdot}\beta^i_rdr + M^c+M^d$ where $M^c$ (resp. $M^d$) is a continuous 
(resp. purely discontinuous) local martingale.   $\int_s^{\cdot}\beta^i_rdr$ is bounded and $M^c$ is a continuous local martingale with bounded bracket $\langle M^c\rangle =[M^c]=\int_s^{\cdot}(\sigma\sigma^{\intercal})^i_r dr$
 hence by BDG inequality, $\underset{t\in[s,T]}{\text{sup }}|M^c_t|\in \mathcal{L}^p(\mathbbm{P}^{s,\eta})$. In order to conclude,
 we are therefore left to show that the same holds for $M^d$. We have $M^d_t=\underset{r\leq t}{\sum}\Delta X^i_r - 
\int_s^t\int_{\mathbbm{R}^d}\gamma^i_r(\cdot,y)F(dy)dr$,
where the integral in previous formula
 is bounded, because $\gamma$ is bounded.
 So we need to show that $\left(\underset{r\leq T}{\sum}|\Delta X^i_r|\right)\in\mathcal{L}^p(\mathbbm{P}^{s,\eta})$.
 Observing the definition of $\mathbbm{P}^{s,\eta}$ and $\tilde{X}$ in Definition \ref{D213} it is clear that since $\gamma$ is bounded then the jumps of $X$ under $\mathbbm{P}^{s,\eta}$ are bounded. So finally, we are left to show that the number of jumps of $X^i$ $\underset{r\leq T}{\sum}\mathds{1}_{\Delta X^i_r\neq 0}$ belongs to $\mathcal{L}^p(\mathbbm{P}^{s,\eta})$. This holds since $X$ can jump only if the underlying Poisson measure $p$ jumps (see Definition \ref{D213}) and the number of jumps of $p$ on $[s,T]$ is a Poisson r.v. of parameter $(T-s)F(\mathbbm{R}^d)$ which admits a finite $p$-th moment.

\end{proof}

\begin{notation}\label{DA} 
	We set $\mathcal{D}(A)$ to be the space of real valued functionals $\Phi\in\mathbbm{C}^{1,2}_b(\Lambda_T)$ such that $\Phi,D\Phi,\nabla\Phi,\nabla^2\Phi$ have polynomial growth and such that $\nabla \Phi$ has the horizontal local Lipschitz property. We define the map $A$ on  $\mathcal{D}(A)$ by setting for every $\Phi\in\mathcal{D}(A)$
	\begin{equation}\label{DefA}
	\begin{array}{rl}
	(A\Phi)_t(\omega):=&D\Phi_t(\omega) + \frac{1}{2}Tr((\sigma\sigma^{\intercal})_t(\omega) \nabla ^2 \Phi_t(\omega))  + \beta_t(\omega)\cdot\nabla  \Phi_t(\omega)\\
	&+\int_{\mathbbm{R}^d}(\Phi_t(\omega+\gamma_t(\omega,y)\mathds{1}_{[t,+\infty[})-\Phi_t(\omega)- \gamma_t(\omega,y)\cdot \nabla  \Phi_t(\omega))F(dy).
	\end{array}
	\end{equation}
We also set $M[\Phi]$ as in \eqref{MafPhi} in Notation \ref{N514}. It
defines an AF.
\end{notation}

\begin{remark} \label{RX}
By Lemma \ref{L69}, and expression \eqref{DefA}, the
	coordinates $X^i, 1 \le i  \le d$ of the canonical process belong to $\mathcal{D}(A)$
	and $A(X^i) =\beta^i$.
\end{remark}

 Definition \ref{poly}, Lemma \ref{Moments}, Notation \ref{DA}  
and the fact that $\beta,\sigma,\gamma$ are bounded and $F$ is finite
  yield the following.
\begin{corollary}\label{Cpoly}
\begin{enumerate}\
\item 
for  every $\Phi\in\mathcal{D}(A)$, 
$A\Phi$ and $\Phi -\int_0^{\cdot}(A\Phi)_rdr$ have polynomial growth;
\item every $\Phi$ with polynomial growth verifies that $\underset{r\in[s,T]}{\text{sup }}|\Phi_r|\in \mathcal{L}^p(\mathbbm{P}^{s,\eta})$
   for all finite $p\ge 1 $, $(s,\eta)\in[0,T]\times\Omega$;
\item for all $\Phi\in\mathcal{D}(A)$, $(s,\eta)\in[0,T]\times \Omega$ and finite $p \ge 1$, we have \\
$\underset{t\in[s,T]}{\text{sup }}\left|\Phi_t-\Phi_s(\eta)-\int_s^tA(\Phi)_rdr\right|\in \mathcal{L}^p(\mathbbm{P}^{s,\eta})$;
\item $(\mathcal{D}(A),A)$ verifies Hypothesis \ref{HypDA};
\end{enumerate}
\end{corollary}

%
%

\begin{proposition}\label{MPSDEeq}
	Let $(s,\eta)\in[0,T]\times\Omega$. A probability measure
 $\mathbbm{P}$ on $(\Omega,\mathcal{F})$ is a weak solution of the SDE with coefficients $\beta,\sigma,\gamma$ starting in $(s,\eta)$ iff 
it solves the martingale problem associated to  $(\mathcal{D}(A),A)$, 
defined in Notation \ref{DA}, in the sense of
 Definition \ref{MPop}.
\end{proposition} 
\begin{proof}
We fix $(s,\eta)$. Let $\mathbbm{P}$ be a weak solution of the SDE with
coefficients $\beta,\sigma,\gamma$ starting in $(s,\eta)$.
We show that it fulfills the martingale problem in the sense of Definition 
\ref{MPop}. 
By Proposition \ref{SDEeq},  we immediately see that 
$\mathbbm{P}(\omega^s=\eta^s)=1$ which constitutes item 1.
of Definition \ref{MPop}.
 By Proposition \ref{SDEeq} it follows 
 that, under  $\mathbbm{P}$, $(X_t)_{t\in[s,+\infty[}$ is a  semi-martingale 
with characteristics $\int_s^{\cdot}\beta_rdr$, $\int_s^{\cdot}(\sigma\sigma^{\intercal})_rdr$ and $\nu:(\omega,C)\mapsto\int_s^{+\infty}\int_E\ \mathds{1}_C(r,\gamma_r(\omega,y))\mathds{1}_{\{\gamma_r(\omega,y)\neq 0\}}F(dy)dr$. 

Now let $\Phi\in\mathcal{D}(A)$. Since for every $\omega$, the set of jump times $\{t:\Delta\omega(t)\neq 0\}$ is countable hence Lebesgue negligible, then
$ \Phi_r(\omega^r) =  \Phi_r(\omega^{r-}), \ dr \ {\rm a.e.},$
and so
 $$\int_s^{\cdot}\int_{\mathbbm{R}^d}(\Phi_r(\omega+\gamma_r(\omega,y)\mathds{1}_{[r,+\infty[})-\Phi_r(\omega)- \gamma_r(\omega,y)\cdot \nabla  \Phi_r(\omega))F(dy)dr,$$ is indistinguishable from $$\int_s^{\cdot}\int_{\mathbbm{R}^d}(\Phi_r(\omega^{r-}+\gamma_r(\omega,y)\mathds{1}_{[r,+\infty[})-\Phi_t(\omega^{r-})- \gamma_t(\omega,y)\cdot \nabla  \Phi_t(\omega^{r-}))F(dy)dr,$$
 which is the compensator 
of 
$$\underset{r\in]s,\cdot]:\Delta X_r\neq 0}{\sum}\Phi_r(\omega)-\Phi_r(\omega^{r-})-  \Delta X_r\cdot\nabla\Phi_r(\omega^{r-}),$$
i.e. their difference is a local martingale.
 By Theorem \ref{ItoFunct}, we therefore have that,
$\Phi-\int_s^{\cdot}A\Phi_rdr $ is a local martingale, and by
item 3. of Corollary \ref{Cpoly} it is a martingale.
 Since this holds for any $\Phi\in\mathcal{D}(A)$, $\mathbbm{P}$ also verifies
 item 2. of Definition \ref{MPop}. 
This concludes the proof of the direct implication.

As far as the converse implication is concerned,
 let us assume that $\mathbbm{P}$ satisfies both items of 
Definition \ref{MPop}.
By Lemma \ref{L69} and Remark \ref{R59}, we have the following. For any $h\in\mathcal{C}_b^{2}(\mathbbm{R}^d)$, the functional $H:(t,\omega)\mapsto h(\omega(t\wedge T))$ belongs to $\mathcal{D}(A)$ and, for any $(t,\omega)\in\Lambda$,  $AH(t,\omega)=A_th(\omega)$, see Notation \ref{NotA}. Definition \ref{MPop} therefore implies that $\mathbbm{P}$ verifies item 3. in Proposition \ref{SDEeq}, hence that it is a weak solution of the SDE.
\end{proof}
	
\begin{corollary} \label{C516}
Let $(\mathbbm{P}^{s,\eta})_{(s,\eta)\in\mathbbm{R}_+\times \Omega}$ be the family
introduced at the beginning of Section \ref{S3b}. 
We suppose the validity  Hypothesis \ref{HypWellPosed}.
\begin{enumerate}
 \item
$(\mathbbm{P}^{s,\eta})_{(s,\eta)\in\mathbbm{R}_+\times \Omega}$ solves the martingale problem associated to $(\mathcal{D}(A),A)$, 
see Definition \ref{MPop}.
\item
 $(\mathcal{D}(A),A)$ is a weak generator of $(P_s)_{s\in\mathbbm{R}_+},$ which is the unique path-dependent system of projectors for which this holds.
\item  For all $\Phi\in\mathcal{D}(A)$, the AF $M[\Phi]$  is  
a square integrable MAF.
\end{enumerate}
\end{corollary}
\begin{proof}
The first statement follows by Proposition \ref{MPSDEeq}, 
the second one by the second statement of Proposition \ref{MPopWellPosed}.
The third statement holds because of Proposition \ref{CoroMaf}. We are indeed in the framework of Section \ref{S2b}, see
Notation \ref{N514}.
\end{proof}

\begin{proposition} \label{P621} 
	Let $\Phi\in \mathcal{D}(A)$ be such that for all $i\leq d$,
 $\Phi X^i\in\mathcal{D}(A)$. Then for all $(t,\omega)\in\mathbbm{R}_+\times\Omega$, 
	\begin{equation}
		\Gamma(X,\Phi)_t(\omega)=(\sigma\sigma^{\intercal}\nabla \Phi)_t(\omega) + \int_{\mathbbm{R}^d}\gamma_t(\omega,y)(\Phi_t(\omega+\gamma_t(\omega,y)\mathds{1}_{[t,+\infty[})-\Phi_t(\omega))F(dy).
	\end{equation}
	
\end{proposition}
\begin{proof}
	We fix $\Phi$ and $1\leq i\leq d$. We recall that the usual product rules apply to both the horizontal and the vertical derivative so that
	
	\begin{equation} 
	\begin{array}{ll} 
	&A(\Phi X^i)_t-\Phi_t AX^i_t - X^i_tA\Phi_t\\
	=&D(\Phi X^i)_t -\Phi_t D X^i_t -X^i_tD\Phi_t\\
	&+ \frac{1}{2}Tr((\sigma\sigma^{\intercal} \nabla^2(\Phi X^i))_t)-\frac{1}{2}\Phi_tTr((\sigma\sigma^{\intercal} \nabla^2X^i)_t) -\frac{1}{2}X^i_tTr((\sigma\sigma^{\intercal} \nabla^2\Phi)_t)\\
	& + \beta_t\cdot \nabla(\Phi X^i)_t-\Phi_t \beta_t\cdot \nabla X^i_t-X^i_t \beta_t\cdot \nabla\Phi_t\\
	&+\int_{\mathbbm{R}^d}(\Phi_t(\cdot+\gamma_t(\cdot,y)\mathds{1}_{[t,+\infty[})(X^i_t+\gamma^i_t(\cdot,y))-\Phi_tX^i_t- \gamma_t(\cdot,y)\cdot \nabla (X^i\Phi)_t)F(dy)\\
	&-\Phi_t\int_{\mathbbm{R}^d}(\gamma^i_t(\cdot,y)- \gamma_t(\cdot,y)\cdot \nabla  \Phi_t)F(dy)\\	
	&-X^i_t\int_{\mathbbm{R}^d}(\Phi_t(\cdot+\gamma_t(\cdot,y)\mathds{1}_{[t,+\infty[})-\Phi_t- \gamma_t(\cdot,y)\cdot \nabla  \Phi_t)F(dy)\\
	=&\frac{1}{2}Tr\left(\sigma\sigma^{\intercal}_t( \nabla^2(\Phi X^i)_t-\Phi_t\nabla^2X^i_t-X^i_t\nabla^2\Phi_t)\right)\\
	&+\int_{\mathbbm{R}^d}(\Phi_t(\cdot+\gamma_t(\cdot,y)\mathds{1}_{[t,+\infty[})(X^i_t+\gamma^i_t(\cdot,y))-\Phi_tX^i_t\\
	&-X^i_t(\Phi_t(\cdot+\gamma_t(\cdot,y)\mathds{1}_{[t,+\infty[})-\Phi_t)-\Phi_t\gamma^i_t(\cdot,y))F(dy)\\
	=& \frac{1}{2}\underset{j,k}{\sum}(\sigma\sigma^{\intercal})_t^{j,k}(\nabla^2_{j,k}(\Phi X^i)_t-\Phi_t\nabla^2_{j,k}X^i_t-X^i_t\nabla^2_{j,k}\Phi_t)\\
	&+\int_{\mathbbm{R}^d}\gamma^i_t(\cdot,y)(\Phi_t(\cdot+\gamma_t(\cdot,y)\mathds{1}_{[t,+\infty[})-\Phi_t)F(dy)\\
	=&\frac{1}{2}\underset{j,k}{\sum}(\sigma\sigma^{\intercal})_t^{j,k}(\nabla_j\Phi_t\nabla_kX^i_t+\nabla^k\Phi_t\nabla^jX^i_t)\\
	&+\int_{\mathbbm{R}^d}\gamma^i_t(\cdot,y)(\Phi_t(\cdot+\gamma_t(\cdot,y)\mathds{1}_{[t,+\infty[})-\Phi_t)F(dy)\\
	=&\frac{1}{2}\left(\underset{j}{\sum}(\sigma\sigma^{\intercal})_t^{j,i}\nabla_j\Phi_t+\underset{k}{\sum}(\sigma\sigma^{\intercal})_t^{i,k}\nabla_k\Phi_t\right)\\
	&+\int_{\mathbbm{R}^d}\gamma_t^i(\cdot,y)(\Phi_t(\cdot+\gamma_t(\cdot,y)\mathds{1}_{[t,+\infty[})-\Phi_t)F(dy)\\
	=&(\sigma\sigma^{\intercal}\nabla\Phi)_t^i+\int_{\mathbbm{R}^d}\gamma^i_t(\cdot,y)(\Phi_t(\cdot+\gamma_t(\cdot,y)\mathds{1}_{[t,+\infty[})-\Phi_t)F(dy),
	\end{array}
	\end{equation}
	where by Lemma \ref{L69}, the fifth equality holds since $\nabla_jX^i$ is constantly equal to $1$ if $j=i$ and $0$ otherwise.
\end{proof}

\begin{proposition}\label{PropN}
	$X$ verifies Hypothesis \ref{HypN2}.
\end{proposition}

\begin{proof}
	We fix $i\leq d$. By Remarks \ref{RX}, $X^i\in\mathcal{D}(A)$ with $A(X^i)=\beta^i$.  A consequence of Lemma \ref{Moments} and of the fact that $\beta$ is bounded is that $X^i$ verifies item 1. of Hypothesis \ref{HypN2}. By Remark \ref{R59} and since $(X^i)^2$ clearly has polynomial growth we have $(X^i)^2\in\mathcal{D}(A)$ which is item 2. of Hypothesis \ref{HypN2}. Finally by Proposition \ref{P621}, we have $\Gamma(X^i)_t(\omega)=(\sigma\sigma^{\intercal})^{i,i}_t(\omega) + 
	\int_{\mathbbm{R}^d}(\gamma^i)^2_t(\omega,y)F(dy)$ which is bounded being the coefficients  $\sigma,\gamma, F$ bounded; so item 3. of Hypothesis \ref{HypN2} is verified.
\end{proof}
From now on we fix $\Psi^i:=X^i$ for all $i$ and the corresponding driving MAF $M[X]$  and we will apply the results of Subsection \ref{S2b} to this specific setup.

We now fix $\xi,f$ verifying  Hypothesis \ref{HypBSDE}. 
We will be interested in the following path-dependent non linear IPDE with terminal condition, denoted by $IPDE(f,\xi)$:
\begin{equation}\label{PDPDE}
\left\{
\begin{array}{l}
	A(\Phi) + f(\cdot,\cdot,\Phi,\Gamma(\Phi,X))=0\text{ on }[0,T]\times\Omega\\
	\Phi_T=\xi\text{ on }\Omega,
\end{array}\right.
\end{equation}
where the explicit expression of $\Gamma(\Phi,X)$ is given by Proposition \ref{P621}.

\begin{remark}
	When $\gamma\equiv 0$, the equation \eqref{PDPDE} is given by 
	\begin{equation}
	\left\{
	\begin{array}{l}
	D\Phi + \frac{1}{2}Tr(\sigma\sigma^{\intercal} \nabla^2\Phi)  + \beta\nabla  \Phi + f(\cdot,\cdot,\Phi,\sigma\sigma^{\intercal} \nabla \Phi )=0\text{ on }[0,T]\times\Omega\\
	\Phi_T=\xi\text{ on }\Omega,
	\end{array}\right.
	\end{equation}
	and it is a path-dependent PDE.
\end{remark}

To the path-dependent IPDE \eqref{PDPDE}, we will associate a family of BSDEs driven by a cadlag martingale, indexed by $(s,\eta)\in[0,T]\times\Omega$.
\begin{notation}
$BSDE^{s,\eta}(f,\xi)$ will denote equation
	\begin{equation}\label{BSDESDE}
	Y^{s,\eta}=\xi+\int_{\cdot}^{T}f\left(r,\cdot,Y^{s,\eta}_r,\frac{d\langle M^{s,\eta},M^{s,\eta}[X]\rangle_r}{dr}\right)dr-(M^{s,\eta}_T-M^{s,\eta}_{\cdot}),
	\end{equation}
	in the stochastic basis $(\Omega,\mathcal{F}^{s,\eta},\mathbbm{F}^{s,\eta},\mathbbm{P}^{s,\eta})$.
\end{notation}

By Proposition \ref{PropN}, $\Psi:= X$, where
$X$ is the  canonical process,
verifies Hypothesis \ref{HypN2};
by item 1. of Remark \ref{RN2}, 
  $M[X]$ satisfies Hypothesis \ref{HypN}.
 Now $\xi,f$ verify Hypothesis \ref{HypBSDE}; so by Theorem \ref{T45},
 for every  $(s,\eta)\in[0,T]\times\Omega$, $BSDE^{s,\eta}(f,\xi)$ admits a unique solution $(Y^{s,\eta},M^{s,\eta})$ in $\mathcal{L}^2(dt\otimes d\mathbbm{P}^{s,\eta})\times\mathcal{H}^2_0(\mathbbm{P}^{s,\eta})$.

\begin{definition}\label{DefClasical}
	Let $\Phi\in\mathcal{D}(A)$ such that $\Phi X^i\in\mathcal{D}(A)$ for all $i\leq d$. We will say that $\Phi$ is a \textbf{classical solution} of $IPDE(f,\xi)$ if it verifies \eqref{PDPDE}.
\end{definition}

We can now formulate the main result of this paper.
\begin{theorem}\label{MainTheorem}
Assume that Hypothesis \ref{HypWellPosed} holds, that $\xi,f(\cdot,\cdot,0,0)$ are Borel with polynomial growth and that $f$ is Lipschitz in $y,z$ uniformly in $t,\omega$. 
\begin{enumerate}
\item The identification problem $IP(f,\xi)$ (see Definition \ref{Abmildsoluv})
 admits a unique solution $(Y,Z)\in\mathcal{L}^2_{uni}\times (L^2_{uni})^d$.
\item 
 $IPDE(f,\xi)$ admits a unique decoupled mild solution $Y$ in the sense that 
whenever $Y$ and $\bar Y$ are two decoupled mild solutions then 
$Y$  and $\bar Y$ are identical. 
\item 
If for every $(s,\eta)$,
  $(Y^{s,\eta},M^{s,\eta})$ is the solution of $BSDE^{s,\eta}(f,\xi)$, i.e. 
\eqref{BSDESDE}, 
 then the decoupled mild solution
$Y$ is given by $(s,\eta) \mapsto Y_s^{s,\eta}$. Moreover,  
 for every $(s,\eta)$, on $[s,T]$, 
$Y^{s,\eta}$ is a $ \mathbbm{P}^{s,\eta}$ version of $Y$ and
$Z_t = \frac{d\langle  M^{s,\eta},M[X]^{s,\eta}\rangle_t}{dt},$ $dt \otimes\mathbbm{P}^{s,\eta}$ a.e.
\end{enumerate}
\end{theorem}

\begin{proof}
It is a consequence of Theorem \ref{AbstractTheorem}. Indeed
firstly  Hypothesis \ref{HypN2} holds
	by Proposition \ref{PropN} and $\Psi:=X$ satisfies
 Hypothesis \ref{HypN2}; secondly  
  $\xi$ and $f(\cdot,\cdot,0,0)$, being  
of polynomial growth 
then $(f,\xi)$ fulfill Hypothesis
 \ref{HypBSDE}
because of item 2. of  Corollary \ref{Cpoly} and
the Lipschitz property of $f$ in $(y,z)$.

\end{proof}

\begin{remark} \label{R523}
As anticipated in the introduction,
given the family of solutions $(Y^{s,\eta}, M^{s,\eta})$ 
of $BSDE^{s,\eta}(f,\xi)$, 
 we have obtained an
 analytical characterization of the process 
$\frac{d\langle  M^{s,\eta},M[X]^{s,\eta}\rangle}{dt}$.
This constitutes indeed the ''identification'' of 
the ''second component'' of a solution to a BSDE via
solving an analytical problem.
\begin{enumerate}
\item  Indeed, by item 3. of Theorem \ref{MainTheorem}, that process is 
$dt \otimes d\mathbbm{P}^{s,\eta}$ a.e.
equal to the q.s. unique functional $Z$  such that $(Y,Z)$ fulfills \eqref{MildEq}. 
\item If $\Gamma(Y,X)$ (hence $\sigma\sigma^{\intercal}\nabla Y$ if $\gamma\equiv 0$)
is well-defined then $(Y,\Gamma(Y,X))$ 
fulfills equation \eqref{MildEq},
 see item 3. of Proposition
\ref{classical}. 
 \item Previous analytical characterization of
$\frac{d\langle  M^{s,\eta},M[X]^{s,\eta}\rangle}{dt}$
 is not possible  with the theory of viscosity solutions,
 even in the case of classical Pardoux-Peng Markovian Brownian BSDEs.
\end{enumerate}

\end{remark}

The link between decoupled mild solutions and classical solutions is the following.
\begin{proposition}\label{classical}
	\begin{enumerate}\
		\item Let $\Phi$ be a classical solution of $IPDE(f,\xi)$, 
see Definition \ref{DefClasical}. Then $(\Phi,\Gamma(\Phi,X))$ is a solution of the identification problem $IP(f,\xi)$ (see Definition \ref{Abmildsoluv}) and in particular, $\Phi$ is a decoupled mild solution of $IPDE(f,\xi)$;
		\item there is at most one classical solution of $IPDE(f,\xi)$;
		\item assume that the unique decoupled mild solution $Y$ of $IPDE(f,\xi)$ verifies $Y\in\mathcal{D}(A)$ and $YX^i\in\mathcal{D}(A),\quad 0\leq i\leq d$, then $Y$ is a classical solution q.s., in the sense that $Y_T=\xi$ (for all $\omega$) and that $A(Y)=-f(\cdot,\cdot,Y,\Gamma(Y,X))$ q.s., see
 Definition \ref{zeropotential}.
	\end{enumerate}
\end{proposition}

\begin{proof}
\begin{enumerate}
\item	
 Let $\Phi$ be a classical solution. First, 
since $\Phi$ and $\Phi X^i$  belong to $\mathcal{D}(A)$ for all $i\leq d$;
taking into account Notation \ref{NotGamma},
by items 1. 2. of Corollary \ref{Cpoly} 
we can show that
$\Phi,\Gamma(\Phi,X^1),\cdots,\Gamma(\Phi,X^d)$ belong to $\mathcal{L}^2_{uni}$. \\
\\
On the other hand, let  $(s,\eta)\in[0,T]\times\Omega$.
By Corollary \ref{C516} 3.
$M[\Phi]_{s,\cdot}  := \Phi-\Phi_s(\eta)-\int_s^{\cdot}A\Phi_rdr,$ and
 $M[\Phi X^i]_{s,\cdot}:= \Phi X^i-\Phi_s(\eta)\eta^i(s)-\int_s^{\cdot}A(\Phi X^i)_rdr,\quad 1\leq i\leq d$, 
 are $\mathbbm{P}^{s,\eta}$-martingales on $[s,T]$ 
vanishing at time $s$. 
By Definition \ref{DefClasical} we have $A\Phi=-f(\cdot,\cdot,\Phi,\Gamma(\Phi,X))$ and by Remark \ref{RX}
and Notation \ref{NotGamma}
we have 
$$A(\Phi X^i)=\Gamma(\Phi X^i)+\Phi AX^i+X^iA\Phi=\Gamma(\Phi X^i)+\Phi\beta^i-X^if(\cdot,\cdot,\Phi,\Gamma(\Phi,X)),$$
 so the previously mentioned martingales indexed by $[s,T]$, 
can be rewritten
\small
\begin{equation}\label{PreviousMart}
\left\{\begin{array}{rcl}
M[\Phi]_{s, \cdot} &=&  \Phi-\Phi_s(\eta)+\int_s^{\cdot}f(r,\cdot,\Phi_r,\Gamma(\Phi,X)_r)dr\\
M[\Phi X^i]_{s,\cdot} &=& 
  \Phi X^i-\Phi_s(\eta)\eta^i(s)-\int_s^{\cdot}(\Gamma(\Phi X^i)_r+\Phi_r\beta_r^i-X^i_rf(r,\cdot,\Phi,\Gamma(\Phi,X)_r))dr,\\
  &&1\leq i\leq d.
\end{array}\right.
\end{equation}
 \normalsize
 Finally, again by Definition \ref{DefClasical} we have  $\Phi_T=\xi$, 
so, for any $(s,\eta)$,
 taking the expectations in \eqref{PreviousMart} at $s = T$, we get  
	\small
	\begin{equation}
	\left\{
	\begin{array}{l}
	\mathbbm{E}^{s,\eta}\left[\xi-\Phi_s(\eta)+\int_s^{T}f(r,\cdot,\Phi_r,\Gamma(\Phi,X)_r)dr\right]=0;\\
	\mathbbm{E}^{s,\eta}\left[\xi X^i_T-\Phi_s(\eta)\eta^i(s)-\int_s^{T}(\Gamma(\Phi X^i)_r+\Phi_r\beta_r^i-X^i_rf(r,\cdot,\Phi,\Gamma(\Phi,X)_r))dr\right]=0,\\
	1\leq i\leq d,
	\end{array}
	\right.
	\end{equation}
	\normalsize
	which by Fubini's Lemma and Definition \ref{ProbaOp} yields
	\small
	\begin{equation}
	\left\{
	\begin{array}{rl}
	\Phi_s(\eta)=&P_s[\xi](\eta)+\int_s^TP_s\left[f\left(r,\cdot,\Phi_r,\Gamma(\Phi,X)_r\right)\right](\eta)dr\\
	(\Phi X^1)_s(\eta) =&P_s[\xi X^1_T](\eta) -\int_s^TP_s\left[\left(\Gamma(\Phi X^1)_r+\Phi_r\beta^1_r-X^1_rf\left(r,\cdot,\Phi_r,\Gamma(\Phi,X)_r\right)\right)\right](\eta)dr\\
	\cdots&\\
	(\Phi X^d)_s(\eta) =&P_s[\xi X^d_T](\eta) -\int_s^TP_s\left[\left(\Gamma(\Phi, X^d)_r+\Phi_r\beta^d_r-X^d_rf\left(r,\cdot,\Phi_r,\Gamma(\Phi,X)_r\right)\right)\right](\eta)dr,
	\end{array}\right.
	\end{equation}
	\normalsize
	and the first item is proven.
\item
	The second item follows from item 1. and by the uniqueness of a decoupled mild solution of $IPDE(f,\xi)$, see Theorem \ref{MainTheorem}. 
	\item
	Concerning item 3. let $(Y,Z)$ be the unique decoupled mild solution of $IP(f,\xi)$. We first note that the first line of \eqref{MildEq} taken with $s=T$ yields $Y_T=\xi$. 
	
	Let us now fix some $(s,\eta)\in[0,T]\times\Omega$. The fact that $Y\in\mathcal{D}(A)$ implies by Proposition \ref{MPSDEeq}  that
 $Y-Y_s-\int_s^{\cdot}AY_rdr$ is on $[s,T]$ a $\mathbbm{P}^{s,\eta}$-martingale and by Theorem \ref{ItoFunct} that this martingale, which we shall denote $M^{s,\eta}[Y]$, is $\mathbbm{P}^{s,\eta}$ a.s. cadlag. Hence  $Y$ is under $\mathbbm{P}^{s,\eta}$ a cadlag special semi-martingale.
Let us keep in mind the solution  $(Y^{s,\eta}, M^{s,\eta})  $ 	
of \eqref{BSDESDE}.
A consequence of item 3. of 
Theorem \ref{MainTheorem}
is that $Y$ admits on $[s,T]$, $Y^{s,\eta}$ 
as $\mathbbm{P}^{s,\eta}$ cadlag version  
 which is a special semi-martingale verifying  
$ Y^{s,\eta}_t = Y_s^{s,\eta} - \int_s^tf(r,Y_r,Z_r)dt + M^{s,\eta}_t, t \in [s,T]$.
 Since $Y$ is $\mathbbm{P}^{s,\eta}$ a.s. cadlag, then $Y$ and $Y^{s,\eta}$ are actually $\mathbbm{P}^{s,\eta}$-indistinguishable on $[s,T]$ and by uniqueness of the decomposition of 
the semi-martingale $Y$, we have that $(\int^{\cdot}_sAY_rdr,M^{s,\eta}[Y])$ and $(-\int^{\cdot}_sf(r,Y_r,Z_r)dr,M^{s,\eta})$ are
  $\mathbbm{P}^{s,\eta}$-indistinguishable on $[s,T]$. Since this holds for all
 $(s,\eta)$, by Definition \ref{zeropotential} we have
 $AY=-f(\cdot,\cdot,Y,Z)$ q.s. so we are left to show that $Z=\Gamma(Y,X)$ q.s.
	 
	 We fix again $(s,\eta)$. 
By item 3. of Theorem \ref{MainTheorem},
 $\langle M^{s,\eta},M^{s,\eta}[X]\rangle = \int_s^{\cdot}Z_rdr$. 
 By item 3. of  Corollary \ref{C516} and
 Lemma \ref{bracketGammabis}, $\langle M^{s,\eta}[Y],M^{s,\eta}[X]\rangle = \int_s^{\cdot}\Gamma(Y,X)_rdr$. As we  have remarked above,
 $M^{s,\eta}=M^{s,\eta}[Y]$ so $\int_s^{\cdot}Z_rdr$ and
 $\int_s^{\cdot}\Gamma(Y,X)_rdr$ are $\mathbbm{P}^{s,\eta}$-indistinguishable on $[s,T]$. Since this holds for all $(s,\eta)$, we indeed have by Definition \ref{zeropotential} that $Z=\Gamma(Y,X)$ q.s., and the proof is complete.
\end{enumerate}
\end{proof}

\begin{appendix}

\section{Appendix: some technicalities}

In all the appendix, we are in the framework of Section \ref{S2}.
\begin{lemma}\label{LemmaBorel} 
Let $\tilde{f}\in\mathcal{L}^1_{uni}$. Then 
$\begin{array}{rcl}
(s,\eta)&\longmapsto &\mathbbm{E}^{s,\eta}[\int_s^{T}\tilde{f}_rdV_r]\\
\ [0,T]\times\Omega&\longrightarrow&\mathbbm{R}
\end{array}$ is $\mathbbm{F}^o$-progressively measurable. 
\end{lemma}
\begin{proof}
We fix $T_0\in]0,T]$
 and we will show that on $[0,T_0]\times\Omega$, $(s,\eta)\longmapsto \mathbbm{E}^{s,\eta}[\int_s^T\tilde{f}_rdV_r]$ is $\mathcal{B}([0,T_0])\otimes\mathcal{F}^o_{T_0}$-measurable.
We will start by showing that on $[0,T_0]\times \Omega\times [0,T_0]$, the function 
\\
$k^n:(s,\eta,t)\mapsto \mathbbm{E}^{s,\eta}[\int_{t}^T((-n)\vee \tilde{f}_r\wedge n)dV_r]$  is $\mathcal{B}([0,T_0])\otimes\mathcal{F}^o_{T_0}\otimes\mathcal{B}([0,T_0])$-measurable, where $n\in\mathbbm{N}$. 
\\
Let $t\in[0,T_0]$ be fixed. Then by Remark \ref{Borel}
\\
 $(s,\eta)\mapsto\mathbbm{E}^{s,\eta}[\int_{t}^{T}((-n)\vee \tilde{f}_r\wedge n)dV_r]$ is 
 $\mathcal{B}([0,T_0])\otimes\mathcal{F}^o_{T_0}$-measurable. 
 \\
 Let $(s,\eta)\in[0,T_0]\times \Omega$ be fixed and $t_m\underset{m\rightarrow\infty}{\longrightarrow} t$ be a converging sequence in $[0,T_0]$.
 \\
 Since $V$ is continuous, 
\begin{equation} \label{EFAS}
\int_{t_m}^{T}((-n)\vee \tilde{f}_r\wedge n)dV_r\underset{m\rightarrow\infty}{\longrightarrow}\int_{t}^T((-n)\vee \tilde{f}_r\wedge n)dV_r \ {\rm a.s.}
\end{equation}
This sequence is uniformly bounded by 
$n V_T$, 
 so by dominated convergence theorem, the convergence in \eqref{EFAS} also
 holds under the expectation, so that $t\mapsto \mathbbm{E}^{s,\eta}[\int_{t}^T((-n)\vee \tilde{f}_r\wedge n)dV_r]$ is continuous. By Lemma 4.51 in \cite{aliprantis}, $k^n$ is therefore 
 $\mathcal{B}([0,T_0])\otimes\mathcal{F}^o_{T_0}\otimes\mathcal{B}([0,T_0])$-measurable.
\\
The  composition of $(s,\eta)\mapsto (s,\eta,s)$ with the maps
$k_n$ yields that, for any $n\geq 0$, $\tilde{k}^n:(s,\eta)\longmapsto\mathbbm{E}^{s,\eta}[\int_{s}^T((-n)\vee \tilde{f}_r\wedge n)dV_r]$ is (on $[0,T_0]\times \Omega$) $\mathcal{B}([0,T_0])\otimes\mathcal{F}^o_{T_0}$-measurable. $\tilde{k}^n$ therefore defines an $\mathbbm{F}^o$-progressively measurable process.
 Then by letting $n$ tend to infinity,  $((-n)\vee \tilde{f}\wedge n)$  tends $dV\otimes d\mathbbm{P}^{s,\eta}$ a.e. to $\tilde{f}$ and since we assumed 
$\mathbbm{E}^{s,\eta}[\int_s^T|\tilde{f}_r|dV_r]<\infty$, by dominated convergence, $\mathbbm{E}^{s,\eta}[\int_{s}^T((-n)\vee \tilde{f}_r\wedge n)dV_r]$ tends to $\mathbbm{E}^{s,\eta}[\int_s^T\tilde{f}_rdV_r]$. 
\\
$(s,\eta)\longmapsto \mathbbm{E}^{s,\eta}[\int_s^T\tilde{f}(r,X_r)dV_r]$ is therefore  an $\mathbbm{F}^o$-progressively measurable process as the pointwise limit of the $\tilde{k}^n$ which are  $\mathbbm{F}^o$-progressively measurable processes.
\end{proof}

We recall the following immediate consequence of Fubini's Theorem which corresponds to Lemma 5.13 in \cite{paper1preprint}.
\begin{lemma}\label{ModifImpliesdV}
	Let $\mathbbm{P}$ be a probability measure on $(\Omega,\mathcal{F})$ and  $\phi,\psi$ be two measurable processes. If $\phi$ and $\psi$ are $\mathbbm{P}$-modifications of each other, then they are equal $dV\otimes d\mathbbm{P}$ a.e.
\end{lemma}

The proof of Proposition \ref{Defuv} goes through 
a linearization lemma.
\begin{lemma} \label{L41}
Let $\tilde{f}\in\mathcal{L}^2_{uni}$.
Let, for every $(s,\eta)\in[0,T]\times \Omega$, 
  $(\tilde Y^{s,\eta},\tilde M^{s,\eta})$ be the unique solution of 
\begin{equation}\label{FBSDEftilde}
\tilde Y^{s,\eta}_t = \xi + \int_t^T \tilde f_rdV_r  -(\tilde M^{s,\eta}_T - \tilde M^{s,\eta}_t),\quad t\in[s,T],
\end{equation}
in $\left(\Omega,\mathcal{F}^{s,\eta},\mathbbm{F}^{s,\eta},\mathbbm{P}^{s,\eta}\right)$.
  Then there exists a process $\tilde Y\in\mathcal{L}^2_{uni}$, a  square integrable path-dependent MAF $(\tilde M_{t,u})_{0\leq t\leq u}$
$\tilde Z^1,\cdots,\tilde Z^d\in\mathcal{L}^2_{uni}$,
 such that for all $(s,\eta)\in[0,T]\times \Omega$ the following holds.
  \begin{enumerate}
  	\item $\tilde Y^{s,\eta}$ is  on $[s,T]$ a $\mathbbm{P}^{s,\eta}$-modification of $\tilde Y$;
  	\item $\tilde M^{s,\eta}$ is the cadlag version of $\tilde M$ under $\mathbbm{P}^{s,\eta}$.
\item For each integer
$1 \le i \le d$,
  $\tilde Z^i=\frac{d\langle  \tilde M^{s,\eta},N^{i,s,\eta}\rangle}{dV}$ $dV\otimes\mathbbm{P}^{s,\eta}$ a.e.
\end{enumerate}
\end{lemma}
\begin{remark} \label{LR41}
The existence, for any $(s,\eta)$, of a unique solution
$(\tilde Y^{s,\eta},\tilde M^{s,\eta})$ of  
\eqref{FBSDEftilde} holds because $\xi$ and
  $(t,\omega,y,z)\mapsto \tilde{f}_t(\omega)$ trivially verify the hypothesis of  Theorem \ref{T45}.
\end{remark}

\begin{proof} \
We set  $\tilde Y:(s,\eta)\mapsto \mathbbm{E}^{s,\eta}\left[\xi + \int_s^T \tilde{f}_rdV_r\right]$ 
which is $\mathbbm{F}^o$-progressively measurable by Remark \ref{Borel} and 
   Lemma \ref{LemmaBorel}. Therefore,  
for a fixed $t \in[s,T]$ we have 
$\mathbb {P}^{s,\eta}$-a.s. 
\begin{equation*}
\begin{array}{rcl}
    \tilde Y_t(\omega) &=& \mathbbm{E}^{t,\omega}\left[\xi + \int_t^T \tilde{f}_rdV_r\right]\\
     &=& \mathbbm{E}^{s,\eta}\left[\xi + \int_t^T \tilde{f}_rdV_r\middle|\mathcal{F}_t\right](\omega)\\
     &=& \mathbbm{E}^{s,\eta}\left[\tilde Y^{s,\eta}_t+(\tilde M^{s,\eta}_T-\tilde M^{s,\eta}_t)|\mathcal{F}_t\right](\omega)\\
     &=& \tilde Y^{s,\eta}_t(\omega).
\end{array}
\end{equation*}
The second equality follows by Remark \ref{Borel} and the third one
uses \eqref{FBSDEftilde}.
For every $0\leq t\leq u$ and $\omega\in\Omega$ we set
\begin{equation}
	 \tilde M_{t,u}(\omega):=\left\{\begin{array}{l}
	 \tilde Y_{u\wedge T}(\omega)-\tilde Y_{t\wedge T}(\omega)-\int_{t\wedge T}^{u\wedge T}\tilde{f}_r(\omega)dV_r \text{ if }\int_{t\wedge T}^{u\wedge T}|\tilde{f}(\omega)|_rdV_r<+\infty,\\
	 0\text{ otherwise}.
	 \end{array} \right.
\end{equation}
 For fixed $(s,\eta)$, \eqref{FBSDEftilde}
 implies 
$d\tilde Y^{s,\eta}_r=-\tilde{f}_rdV_r+d\tilde M^{s,\eta}_r$. On the other hand
 $\int_s^T|\tilde{f}|_rdV_r<+\infty$ $\mathbbm{P}^{s,\eta}$ a.s.; so for any $s\leq t\leq u$ we have
 $\tilde M^{s,\eta}_u-\tilde M^{s,\eta}_t=\tilde M_{t,u}$ $\mathbb {P}^{s,\eta}$- a.s.  Taking into account that $\tilde M^{s,\eta}$ is square integrable  and the fact that previous 
equality holds for any $(s,\eta)$ and $t\leq u$, then $(\tilde M_{t,u})_{0\leq t\leq u}$ indeed defines a  square integrable path-dependent MAF. 
$Y$ belongs to $\mathcal{L}^2_{uni}$ 
because of the validity the two following arguments hold for all $(s,\eta)$.
 First $Y$ is a $\mathbbm{P}^{s,\eta}$-modification 
of $Y^{s,\eta}$ on $[s,T]$,  so 
 by  Lemma \ref{ModifImpliesdV}
$ Y = Y^{s,\eta}$
 $dV\otimes\mathbbm{P}^{s,\eta}$ a.e.;
second $Y^{s,\eta}\in\mathcal{L}^2(dV\otimes\mathbbm{P}^{s,\eta})$.
 The existence of $Z$ follows setting for all $i$, $Z^i=\frac{d\langle  \tilde M,N^i\rangle}{dV}$, see Notation \ref{RadonAF} and Proposition \ref{bracketMAFs}.
\end{proof}

\begin{notation} \label{N49}
For every fixed $(s,\eta)\in [0,T]\times \Omega$, we will denote by  $(Y^{k,s,\eta},M^{k,s,\eta})_{k\in\mathbbm{N}}$ the Picard iterations  associated to  $BSDE^{s,\eta}(f,\xi)$ as defined in Notation A.13 in \cite{paper3} and $Z^{k,s,\eta}:=(Z^{k,1,s,\eta},\cdots Z^{k,d,s,\eta})$ will denote $\frac{\langle M^{k,s,\eta},N^{s,\eta}\rangle}{dV}$. \\

This means that for all $(s,\eta)\in[0,T]\times\Omega$, $(Y^{0,s,\eta},M^{0,s,\eta})\equiv (0,0)$ and for all $k\geq 1$, we have on $[s,T]$  
\begin{equation}\label{defYk}
	Y^{k,s,\eta}=\xi+\int_{\cdot}^{T}f(r,\cdot,	Y^{k-1,s,\eta}_r,Z^{k-1,s,\eta}_r)dV_r-(M^{k,s,\eta}_T-M^{k,s,\eta}_{\cdot}),
\end{equation}
in the sense of $\mathbbm{P}^{s,\eta}$-indistinguishability, and that for all $(s,\eta)\in[0,T]\times\Omega$, $k\geq 0$, $Y^{k,s,\eta},Z^{1,k,s,\eta},\cdots Z^{d,k,s,\eta}$ belong to $\mathcal{L}^2(dV\otimes\mathbbm{P}^{s,\eta})$, see Notation A.13 and Lemma A.2 in \cite{paper3}.

\end{notation}
A direct consequence of Proposition A.15 in \cite{paper3} and the lines above it, is the following.
\begin{proposition}\label{cvYZk}
	For every $(s,\eta)\in[0,T]\times\Omega$,
each component of \\ $(Y^{k,s,\eta},Z^{1,k,s,\eta},\cdots Z^{d,k,s,\eta})$ tends to
each component of
 $(Y^{s,\eta},Z^{1,s,\eta},\cdots Z^{d,s,\eta})$ in $\mathcal{L}^2(dV\otimes\mathbbm{P}^{s,\eta})$ and $dV\otimes\mathbbm{P}^{s,\eta}$-a.e. when $k$ tends to infinity.
\end{proposition}

\begin{proposition} \label{P511}
	For each $k\in\mathbbm{N}$,  there exist processes $Y^k\in\mathcal{L}^2_{uni},
  Z^{k,1},\cdots, Z^{k,d}\in\mathcal{L}^2_{uni}$,
 a  square integrable path-dependent MAF $(M^k_{t,u})_{0\leq t\leq u}$ such that for all $(s,\eta)\in[0,T]\times \Omega$, we have the following.
	\begin{enumerate}
	\item $ Y^{k,s,\eta}$ is  on $[s,T]$ a $\mathbbm{P}^{s,\eta}$-modification of $Y^k$;
\item $ M^{k,s,\eta}$ is the cadlag version of $M^k$ under $\mathbbm{P}^{s,\eta}.$
\item For all $(s,\eta)\in[0,T]\times \Omega$ and $i\in[\![ 1;d]\!]$, $ Z^{k,i}=\frac{d\langle   M^{k,s,\eta},N^{i,s,\eta}\rangle}{dV}$ $dV\otimes\mathbbm{P}^{s,\eta}$ a.e.
\end{enumerate}
\end{proposition}

\begin{proof}
We prove the statement by induction on $k \ge 0.$
	It is clear that $Y^0\equiv 0$ and $M^0\equiv 0$ verify the assertion for $k=0$.
\\
Suppose the existence, for $k \ge 1$, of a square integrable  path-dependent MAF $M^{k-1}$ and  processes
 $Y^{k-1}$
 $ Z^{k-1,1},\cdots, Z^{k-1,d}\in\mathcal{L}^2_{uni}$ such that
the statements 1. 2. 3. above hold replacing 
$k$ with $k-1$. \\ 

	We fix $(s,\eta)\in[0,T]\times \Omega$. By Lemma \ref{ModifImpliesdV},
	$(Y^{k-1,s,\eta},Z^{k-1,s,\eta})=(Y^{k-1},Z^{k-1})$ $dV\otimes \mathbbm{P}^{s,\eta}$ a.e. 
Therefore by \eqref{defYk}
$$ Y_t^{k,s,\eta} = \xi + \int_t^T f\left(r,\cdot,Y^{k-1}_r,Z^{k-1}_r\right)dV_r  -(M^{k,s,\eta}_T - M^{k,s,\eta}_t), t \in [s,T].$$ 
According to Notation \ref{N36}, the equation \eqref{defYk} can be seen 
 as a BSDE of the type $BSDE^{s,\eta}(\tilde f,\xi)$ where $\tilde f:(t,\omega)\longmapsto f(t,\omega,Y^{k-1}_t(\omega),Z^{k-1}_t(\omega))$. We now verify that $\tilde f$ verifies the conditions under which Lemma \ref{L41} applies.

$\tilde f$ is
 $\mathbbm{F}^o$-progressively measurable
since $Y^{k-1},Z^{k-1}$ are $\mathbbm{F}^o$-progressively measurable
 and $f$ is $\mathcal{P}ro^o\otimes\mathcal{B}(\mathbbm{R})\otimes\mathcal{B}(\mathbbm{R}^d)$-measurable so $\tilde f$ is $\mathbbm{F}^o$-progressively measurable. 
Since 
$$|\tilde f(t,\omega)|=|f(t,\omega,Y^{k-1}_t(\omega),Z^{k-1}_t(\omega))|\leq |f(t,\omega,0,0)|+K(|Y^{k-1}_t(\omega)|+\|Z^{k-1}_t(\omega)\|),$$
for all $t,\omega$, with $f(\cdot,\cdot,0,0),Y^{k-1},Z^{1,k-1},\cdots,Z^{d,k-1}\in\mathcal{L}^2_{uni}$ by recurrence hypothesis, it is clear that $\tilde f\in\mathcal{L}^2_{uni}$.
Since 	$(Y^{k,s,\eta},M^{k,s,\eta})$  is a solution of
 $BSDE^{s,\eta}(\tilde f,\xi)$, Lemma \ref{L41} shows
the existence of suitable $Y^k,M^k,  Z^{k,1},\cdots, Z^{k,d}$ 
verifying the statement for the integer $k$.
\end{proof}

\begin{prooff}\\ of Proposition \ref{Defuv}.
We  define $\bar{Y}$ and $\bar Z^i,  1 \le i \le d$ by
 ${\bar Y}_s(\eta):= \underset{k\in\mathbbm{N}}{\text{limsup }}Y^k_s(\eta)$ and 
$\bar Z^i_s(\eta):=\underset{k\in\mathbbm{N}}{\text{limsup }}Z^{k,i}_s(\eta),$ 
for every  $(s,\eta)\in[0,T]\times \Omega$. \\
 $\bar{Y}$ and $\bar Z:=(\bar Z^1,\cdots,\bar Z^d)$ are $\mathbbm{F}^o$-progressively measurable. Combining Propositions \ref{P511}, \ref{cvYZk} and Lemma \ref{ModifImpliesdV} it follows that for every $(s,\eta)\in[0,T]\times \Omega$,

\begin{equation} \label{EaeConvergence}
\left\{\begin{array}{rcl}
     Y^k&\underset{k\rightarrow\infty}{\longrightarrow}& Y^{s,\eta}  \quad dV\otimes d\mathbbm{P}^{s,\eta}\\
     Z^{k,i}&\underset{k\rightarrow\infty}{\longrightarrow}& Z^{i,s,\eta}  \quad dV\otimes d\mathbbm{P}^{s,\eta}, \text{ for all }
1 \le i \le d.
\end{array}\right.
\end{equation}
Let us fix $1 \le i \le d$ and   $(s,\eta)\in[0,T]\times \Omega$.
There is a set $A^{s,\eta}$ 
of  full  $dV\otimes d\mathbbm{P}^{s,\eta}$ measure
such that for all $(t, \omega) \in A^{s,\eta}$ we have
\begin{equation}\label{E420}
\left\{\begin{array}{rcccccl}
     \bar Y_t(\omega)&=&\underset{k\in\mathbbm{N}}{\text{limsup }}Y^k_t(\omega)&=&\underset{k\in\mathbbm{N}}{\text{lim }}Y^k_t(\omega) &=& Y^{s,\eta}_t(\omega)  \\
     \bar Z_t(\omega)&=&\left(\underset{k\in\mathbbm{N}}{\text{limsup }} Z^{k,i}_t(\omega)\right)_{i\leq d}&=&\left(\underset{k\in\mathbbm{N}}{\text{lim }} Z^{k,i}_t(\omega)\right)_{i\leq d} &=& Z^{s,\eta}_t(\omega). 
\end{array}\right.
\end{equation}
This implies
\begin{eqnarray} \label{E37bis} 
 \bar Y_t(\omega)&=& Y^{s,\eta} \ dV\otimes d\mathbbm{P}^{s,\eta} {\rm a.e.} \\
 \bar Z_t(\omega)&=& Z^{s,\eta} \ dV\otimes d\mathbbm{P}^{s,\eta} {\rm a.e.} \nonumber
 \end{eqnarray}
By \eqref{E37bis} and \eqref{ET45},
under every $\mathbbm{P}^{s,\eta}$, we actually have
\begin{equation} \label{E421}
Y^{s,\eta} = \xi + \int_{\cdot}^T f\left(r,\cdot,\bar{Y}_r,\bar Z_r\right)dV_r  -(M^{s,\eta}_T - M^{s,\eta}_{\cdot}),
\end{equation}
in the sense of $\mathbbm{P}^{s,\eta}$-indistinguishability,
 on the interval $[s,T]$. 
At this stage, in spite of \eqref{E37bis}, $\bar Y$ is not
necessarily a modification of $Y^{s,\eta}$. We will construct processes $Y,Z$
fulfilling indeed the statement of Proposition \ref{Defuv}.
In particular $Y$ fulfills item 1. that is a bit stronger than
 \eqref{E37bis}.

We set now $\tilde f:(t,\omega)\mapsto f(t,\omega,\bar{Y}_t(\omega),\bar Z_t(\omega))$; 
equation \eqref{E421} is now of the form 
\eqref{FBSDEftilde}
and we show that $\tilde f$ so defined verifies the conditions under which  
Lemma \ref{L41} applies.
$\tilde f$ is  $\mathbbm{F}^o$-progressively measurable 
since $f$ is $\mathcal{P}ro^o\otimes\mathcal{B}(\mathbbm{R})\otimes\mathcal{B}(\mathbbm{R}^d)$-measurable and $\bar Y,\bar Z$ are $\mathbbm{F}^o$-progressively measurable.
\\
Moreover, for any $(s,\eta)\in[0,T]\times \Omega$, $Y^{s,\eta}$ and $Z^{1,s,\eta},\cdots,Z^{d,s,\eta}$ belong to $\mathcal{L}^2(dV\otimes d\mathbbm{P}^{s,\eta})$;
 therefore by \eqref{E420}, so do $\bar{Y}$ and $\bar Z^1,\cdots,\bar Z^d$. 
 
Since this holds for all $(s,\eta)$, then  $\bar{Y}$ and $\bar Z^1,\cdots,\bar Z^d$ belong to $\mathcal{L}^2_{uni}$.
\\
Finally, since $|\tilde f(t,\omega)|=|f(t,\omega,\bar Y_t(\omega),\bar Z_t(\omega))|\leq |f(t,\omega,0,0)|+K(|\bar Y_t(\omega)|+\|\bar Z_t(\omega)\|)$ for all $t,\omega$, with
$f(\cdot,\cdot,0,0),\bar Y,\bar Z^1,\cdots,\bar Z^{d}\in\mathcal{L}^2_{uni}$, it is clear that $\tilde f\in\mathcal{L}^2_{uni}$.
Now \eqref{E421} can be considered as a BSDE where the driver does
not depend on $y$ and $z$ of the form \eqref{FBSDEftilde}.
 We can therefore apply Lemma \ref{L41} 
to $\tilde{f}$ and conclude on the existence of $(Y,M,Z^1,\cdots,Z^d)$ verifying the three items of the proposition. 

It remains to prove now the last assertion of Proposition \ref{Defuv}.
We  fix some $(s,\eta)$. The first item implies  that $Y_s= Y^{s,\eta}_s$ $\mathbbm{P}^{s,\eta}$ a.s. But since $Y_s$ is $\mathcal{F}^o_s$-measurable and $\mathbbm{P}^{s,\eta}(\omega^s=\eta^s)=1$, it also yields that $Y_s$ is $\mathbbm{P}^{s,\eta}$ a.s. equal to the deterministic value $Y_s(\eta)$ hence $Y^{s,\eta}_s$ is $\mathbbm{P}^{s,\eta}$ a.s. equal to the deterministic value $Y_s(\eta)$. 
This also proves that $Y$ is unique because it is given by
 $Y:(s,\eta)\longmapsto Y^{s,\eta}_s$. The uniqueness of $Z$ up to zero potential sets is immediate by the third item of the proposition and Definition \ref{zeropotential}.
\end{prooff}
{\bf ACKNOWLEDGEMENTS.}
The research of the first named author was provided 
by a PhD fellowship (AMX) of the Ecole Polytechnique.
The contribution of the second named author
 was partially supported by the grant 346300 for IMPAN from the Simons Foundation and the matching 2015-2019 Polish MNiSW fund.


\end{appendix}

\bibliographystyle{plain}
\bibliography{../biblioPhDBarrasso_bib/biblioPhDBarrasso}

\def\polhk#1{\setbox0=\hbox{#1}{\ooalign{\hidewidth
  \lower1.5ex\hbox{`}\hidewidth\crcr\unhbox0}}}
  \def\polhk#1{\setbox0=\hbox{#1}{\ooalign{\hidewidth
  \lower1.5ex\hbox{`}\hidewidth\crcr\unhbox0}}} \def\cprime{$'$}
  \def\polhk#1{\setbox0=\hbox{#1}{\ooalign{\hidewidth
  \lower1.5ex\hbox{`}\hidewidth\crcr\unhbox0}}}
\begin{thebibliography}{10}

\bibitem{aliprantis}
C.~D. Aliprantis and K.~C. Border.
\newblock {\em Infinite-dimensional analysis}.
\newblock Springer-Verlag, Berlin, second edition, 1999.
\newblock A hitchhiker's guide.

\bibitem{paper1preprint}
A.~Barrasso and F.~Russo.
\newblock Backward {S}tochastic {D}ifferential {E}quations with no driving
  martingale, {M}arkov processes and associated {P}seudo {P}artial
  {D}ifferential {E}quations.
\newblock 2017.
\newblock Preprint, hal-01431559, v2.

\bibitem{paper2}
A.~Barrasso and F.~Russo.
\newblock Backward {S}tochastic {D}ifferential {E}quations with no driving
  martingale, {M}arkov processes and associated {P}seudo {P}artial
  {D}ifferential {E}quations. part {II}: Decoupled mild solutions and examples.
\newblock 2017.
\newblock Preprint, hal-01505974.

\bibitem{paper3}
A.~Barrasso and F.~Russo.
\newblock Martingale driven {BSDE}s, {PDE}s and other related deterministic
  problems.
\newblock 2017.
\newblock Preprint, hal-01566883.

\bibitem{paperAF}
A.~Barrasso and F.~Russo.
\newblock A note on time-dependent additive functionals.
\newblock {\em Communications on Stochastic Analysis}, 11 no 3:313--334, 9
  2017.

\bibitem{paperMPv2}
A.~Barrasso and F.~Russo.
\newblock Path-dependent {M}artingale {P}roblems and {A}dditive {F}unctionals.
\newblock 2017.
\newblock Preprint.

\bibitem{BionNadal}
J.~Bion-Nadal.
\newblock Dynamic risk reasures and path-dependent second order {PDE}s.
\newblock {\em Stochastics of Environmental and Financial Economics},
  138:147--178, 2016.

\bibitem{bismut}
J.M. Bismut.
\newblock Conjugate convex functions in optimal stochastic control.
\newblock {\em J. Math. Anal. Appl.}, 44:384--404, 1973.

\bibitem{sant}
R.~Carbone, B.~Ferrario, and M.~Santacroce.
\newblock Backward stochastic differential equations driven by c\`adl\`ag
  martingales.
\newblock {\em Teor. Veroyatn. Primen.}, 52(2):375--385, 2007.

\bibitem{contfournie10}
R.~Cont and D.-A. Fourni{\'e}.
\newblock Change of variable formulas for non-anticipative functionals on path
  space.
\newblock {\em J. Funct. Anal.}, 259(4):1043--1072, 2010.

\bibitem{contfournie13}
R.~Cont and D.-A. Fourni{\'e}.
\newblock Functional {I}t\^o calculus and stochastic integral representation of
  martingales.
\newblock {\em Ann. Probab.}, 41(1):109--133, 2013.

\bibitem{cosso_russo15b}
A.~Cosso and F.~Russo.
\newblock \emph{Strong-viscosity solutions: semilinear parabolic {PDE}s and
  path-dependent {PDE}s}.
\newblock 2015.

\bibitem{cosso_russo15a}
A.~Cosso and F.~Russo.
\newblock Functional {I}t\^o versus {B}anach space stochastic calculus and
  strict solutions of semilinear path-dependent equations.
\newblock {\em Infin. Dimens. Anal. Quantum Probab. Relat. Top.},
  19(4):1650024, 44, 2016.

\bibitem{DGR}
C.~Di~Girolami and F.~Russo.
\newblock Infinite dimensional stochastic calculus via regularization and
  applications.
\newblock {\em Preprint \textup{HAL-INRIA, inria-00473947 version 1}},
  (Unpublished), 2010.

\bibitem{dupire}
B.~Dupire.
\newblock {\em Functional {I}t\^o calculus}.
\newblock Portfolio Research Paper, Bloomberg, 2009.

\bibitem{ektz}
I.~Ekren, C.~Keller, N.~Touzi, and J.~Zhang.
\newblock On viscosity solutions of path dependent {PDE}s.
\newblock {\em Ann. Probab.}, 42(1):204--236, 2014.

\bibitem{EthierKurz}
S.~N. Ethier and T.~G. Kurtz.
\newblock {\em Markov processes}.
\newblock Wiley Series in Probability and Mathematical Statistics: Probability
  and Mathematical Statistics. John Wiley \& Sons, Inc., New York, 1986.
\newblock Characterization and convergence.

\bibitem{flandoli_zanco13}
F.~Flandoli and G.~Zanco.
\newblock An infinite-dimensional approach to path-dependent {K}olmogorov
  equations.
\newblock {\em Ann. Probab.}, 44(4):2643--2693, 2016.

\bibitem{masiero}
M.~Fuhrman, F.~Masiero, and G.~Tessitore.
\newblock Stochastic equations with delay: optimal control via {BSDE}s and
  regular solutions of {H}amilton-{J}acobi-{B}ellman equations.
\newblock {\em SIAM J. Control Optim.}, 48(7):4624--4651, 2010.

\bibitem{jacod79}
J.~Jacod.
\newblock {\em Calcul stochastique et probl\`emes de martingales}, volume 714
  of {\em Lecture Notes in Mathematics}.
\newblock Springer, Berlin, 1979.

\bibitem{jacod}
J.~Jacod and A.~N. Shiryaev.
\newblock {\em Limit theorems for stochastic processes}, volume 288 of {\em
  Grundlehren der Mathematischen Wissenschaften [Fundamental Principles of
  Mathematical Sciences]}.
\newblock Springer-Verlag, Berlin, second edition, 2003.

\bibitem{leao_ohashi_simas14}
D.~Le\~ao, A.~Ohashi, and A.~B. Simas.
\newblock {\em \textup{Weak functional {I}t\^o calculus and applications}}.
\newblock \emph{Preprint} arXiv:1408.1423v2, 2014.

\bibitem{qian}
G.~Liang, T.~Lyons, and Zh. Qian.
\newblock Backward stochastic dynamics on a filtered probability space.
\newblock {\em Ann. Probab.}, 39(4):1422--1448, 2011.

\bibitem{parpen90}
{\'E}.~Pardoux and S.~Peng.
\newblock Adapted solution of a backward stochastic differential equation.
\newblock {\em Systems Control Lett.}, 14(1):55--61, 1990.

\bibitem{pardoux_peng92}
{\'E}.~Pardoux and S.~Peng.
\newblock Backward stochastic differential equations and quasilinear parabolic
  partial differential equations.
\newblock In {\em Stochastic partial differential equations and their
  applications ({C}harlotte, {NC}, 1991)}, volume 176 of {\em Lecture Notes in
  Control and Inform. Sci.}, pages 200--217. Springer, Berlin, 1992.

\bibitem{Peng2016}
S.~Peng and F.~Wang.
\newblock {BSDE}, path-dependent {PDE} and nonlinear {F}eynman-{K}ac formula.
\newblock {\em Science China Mathematics}, 59(1):19--36, 2016.

\bibitem{stroock}
D.~W. Stroock and S.~R.~S. Varadhan.
\newblock {\em Multidimensional diffusion processes}.
\newblock Classics in Mathematics. Springer-Verlag, Berlin, 2006.
\newblock Reprint of the 1997 edition.

\end{thebibliography}

\end{document}